\documentclass[12pt]{amsart}
\usepackage{a4wide}
\usepackage{amssymb}
\usepackage{amsmath}
\usepackage{amsfonts}
\usepackage{amsthm}

\usepackage{graphicx}
\graphicspath{{./FiguresEPS/}}

\usepackage{longtable}
\usepackage{paralist}
\usepackage[usenames,dvipsnames]{color}
\usepackage{hyperref}
\usepackage{caption}
\usepackage[labelformat=simple,labelfont={}]{subfig}

\newtheorem{thm}{Theorem}[section]
\newtheorem{lem}[thm]{Lemma}
\newtheorem{cor}[thm]{Corollary}
\newtheorem{prop}[thm]{Proposition}
\newtheorem{conj}[thm]{Conjecture}

\theoremstyle{definition}

\theoremstyle{remark}
\newtheorem*{rmk}{Remark}

\numberwithin{equation}{section}

\newcommand{\eps}{\varepsilon}
\newcommand{\ConstEfd}{B_d}
\newcommand{\Econstdone}{B_{d,1}}
\newcommand{\Econstdp}{B_{d,p}}
\newcommand{\COVconstd}{B_d^\prime}
\newcommand{\COV}{\alpha}
\newcommand{\sepDist}{\vartheta}

\newcommand{\DEF}{{\,:=\,}}
\newcommand{\FED}{{\,=:\,}}

\newcommand{\EulerGamma}{\gamma}

\newcommand{\PT}[1]{\mathbf{#1}}
\newcommand{\PTCfg}{X}
\newcommand{\R}[1]{\mathbb{R}^{#1}}
\newcommand{\Sph}[1]{\mathbb{S}^{#1}}

\DeclareMathOperator{\dd}{\mathrm{d}}

\DeclareMathOperator{\betafcn}{B}
\DeclareMathOperator{\gammafcn}{\Gamma}
\DeclareMathOperator{\gammafcnregularizedP}{P}
\DeclareMathOperator{\digammafcn}{\psi}
\DeclareMathOperator{\IncompleteBetaRegularized}{I}
\DeclareMathOperator{\prob}{{Prob}}
\DeclareMathOperator{\var}{Var}

\DeclareMathOperator{\HyperF}{F}

\newcommand{\Hypergeom}[5]{{\sideset{_#1}{_#2}\HyperF\!\left(\substack{\displaystyle#3\\\displaystyle#4};#5\right)}}

\newcommand{\Pochhsymb}[2]{{\left(#1\right)_{#2}}}

\newcommand{\MARKED}[2]{{\textcolor{black}{#2}}}

\allowdisplaybreaks[1]

\title{Random Point Sets on the Sphere --- Hole Radii, Covering, and Separation}
\author[J.S.~Brauchart, A.B. Reznikov, E.B.~Saff, I.H.~Sloan, Y.G.~Wang and R.S.~Womersley]{J.~S.~Brauchart, A. B. Reznikov, E.~B.~Saff, I.~H.~Sloan, Y.~G.~Wang and R.~S.~Womersley}
\thanks{\noindent This research was supported under Australian Research Council's Discovery Projects funding scheme (project number DP120101816).
The research of J.~S. Brauchart was also supported by the Austrian Science Fund FWF projects F5510 (part of the Special Research Program (SFB) ``Quasi-Monte Carlo Methods: Theory and Applications''). The research of A.~B. Reznikov and E.~B. Saff was also supported by U.S. National Science Foundation grants DMS-1412428 and DMS-1516400. The authors also acknowledge the support of the Erwin Schr{\"o}dinger Institute in Vienna, where part of the work was carried out.}

\date{\today}

\DeclareCaptionLabelFormat{andtable}{#1˜\ #2 and \tablename˜\ \thetable}

\begin{document}

\address{J.~S.~Brauchart:
Institut f\"ur Analysis und Zahlentheorie,
Technische Universit\"at Graz,
Steyrergasse 30,
8010 Graz,
Austria
}
\email{j.brauchart@tugraz.at}
\address{A.~B. Reznikov, E.~B.~Saff:
Center for Constructive Approximation,
Department of Mathematics, 
Vanderbilt University,
Nashville, TN 37240,
USA}
\email{edward.b.saff@vanderbilt.edu, aleksandr.b.reznikov@vanderbilt.edu}
\address{I.~H.~Sloan, Y.~G.~Wang and R.~S.~Womersley: 
School of Mathematics and Statistics,
University of New South Wales,
Sydney, NSW, 2052,
Australia }
\email{i.sloan@unsw.edu.au, yuguang.e.wang@gmail.com, r.womersley@unsw.edu.au}

\begin{abstract}
Geometric properties of $N$ random points distributed independently and uniformly on the unit sphere $\mathbb{S}^{d}\subset\mathbb{R}^{d+1}$ with respect to surface area measure are obtained and several related conjectures are posed.
In particular, we derive asymptotics (as $N \to \infty$) for the expected moments of the radii of spherical caps associated with the facets of the convex hull of $N$ random points on $\mathbb{S}^{d}$.
We provide conjectures for the asymptotic distribution of the scaled radii of these spherical caps
and the expected value of the largest of these radii (the covering radius).
Numerical evidence is included to support these conjectures.
\MARKED{red}{Furthermore, utilizing the extreme law for pairwise angles of Cai et al., we derive precise asymptotics for the expected separation of random points on $\mathbb{S}^{d}$.} 
\end{abstract}


\maketitle


\section{Introduction}

This paper is concerned with geometric 
properties of random points distributed independently and uniformly on the unit sphere
$\Sph{d}\subset\R{d+1}$.
The two most common geometric properties 
associated with a 
configuration $\PTCfg_N = \{ \PT{x}_1, \dots, \PT{x}_N \}$ of distinct points on $\Sph{d}$
are the \emph{covering radius} (also known as \emph{fill radius} or \emph{mesh
norm}),
\begin{equation*}
  \COV( \PTCfg_N ) \DEF \COV( \PTCfg_N; \Sph{d} ) \DEF
  \max_{\PT{y} \in \Sph{d}} \min_{1 \leq j \leq N} \arccos( \PT{y}, \PT{x}_j ),
\end{equation*}
which is the largest geodesic distance from a point in $\Sph{d}$ to the
nearest point in $\PTCfg_N$ (or the geodesic radius of the largest
spherical cap that contains no points from $\PTCfg_N$), and the
\emph{separation distance}
\begin{equation*}
  \MARKED{red}{\sepDist( \PTCfg_N )} \DEF \min_{\substack{1 \leq j, k \leq N \\ j \neq k}} \arccos( \PT{x}_j, \PT{x}_k ),
\end{equation*}
which gives the least geodesic distance between two points in $\PTCfg_N$.
\MARKED{red}{
(For related properties of random geometric configurations on the sphere and in the Euclidean space see, e.g., \cite{ArSa2015,BaGoHe2013,CaChe2014,Mi1970,Mo1989}.)
}

\MARKED{red}{One of our} main contributions in this paper concerns a different but related
quantity, namely the sum of powers of the ``hole radii''.  A point
configuration $\PTCfg_N$ on $\Sph{d}$ uniquely defines a convex polytope, 
namely the convex hull of the point configuration. In
turn, each facet of that polytope defines a ``hole'', which we take to
mean the maximal spherical cap for the particular facet
that contains points of $\PTCfg_N$ only on its boundary. The connection
with the covering problem is that the geodesic radius of the largest hole
is the covering radius $\COV( \PTCfg_N )$.

If the number of facets (or equivalently the number of holes)
corresponding to the point set $\PTCfg_N$ is $f_d$ (itself a random
variable for a random set $\PTCfg_N$), then the holes can be labeled from $1$ to $f_d$.  It turns out
to be convenient to define the  $k$th hole radius $\rho_k = \rho_k( \PTCfg_N )$ to be the
\emph{Euclidean} distance in $\R{d+1}$ from the cap boundary to the center
of the spherical cap \MARKED{red}{located on the sphere} ``above'' the $k$th facet,
so $\rho_k = 2 \sin(\alpha_k/2)$, where $\alpha_k$ is the geodesic radius of the cap.

For arbitrary $p\ge 0$, our results concern the sums of the $p$th powers of
the hole radii,
\begin{equation*}
\sum_{k=1}^{f_d} (\rho_k)^p.
\end{equation*}
It is clear that for large $p$ the largest holes dominate, and that
\begin{equation*}
  \lim_{p\to\infty} \left(\sum_{k=1}^{f_d} (\rho_k)^p\right)^{1/p} =
  \max_{1 \leq k \leq f_d} \rho_k \FED \rho(\PTCfg_N) = 2 \sin(\COV( \PTCfg_N ) / 2 ),
\end{equation*}
using the conversion from the geodesic radius $\COV(\PTCfg_N)$ to
the Euclidean radius of the largest hole.

To state our result for the expected moments of the hole radii \MARKED{red}{(Theorem~\ref{thm:sums.of.powers.of.holes.radii})}, we utilize the following notation dealing with random polytopes.
Let $\omega_d$ be the surface area of $\Sph{d}$, so $\omega_d = 2 \pi^{(d+1) /2} / \Gamma((d+1)/2)$, and define (using $\omega_0 = 2$)
\begin{equation} \label{eq:Sdconst}
\kappa_d \DEF \frac{1}{d} \frac{\omega_{d-1}}{\omega_d} =
\frac{1}{d} \frac{\Gamma((d+1)/2)}{\sqrt{\pi}\,\Gamma(d/2)},
\qquad
\ConstEfd \DEF \frac{2}{d+1}\; \frac{\kappa_{d^2}}{(\kappa_d)^d}, \qquad d = 1, 2, 3, \ldots.
\end{equation}
The quantity $\kappa_d$ can also be defined recursively by
\begin{equation*}
\kappa_1 = \frac{1}{\pi}, \qquad
\kappa_d = \frac{1}{2\pi d \kappa_{d-1}}, \quad d = 2,3,\ldots.
\end{equation*}
From \cite{BuMuTi1985}, the expected number of facets\footnote{See also K. Fukuda, %
Frequently asked questions about polyhedral computation,
Swiss Federal Institute of Technology,
\url{http://www.inf.ethz.ch/personal/fukudak/polyfaq/polyfaq.html},
accessed August 2016.} formed from $N$ random points independently and uniformly distributed on $\Sph{d}$
is 
\begin{equation} \label{eq:Buchta.et.al}
\mathbb{E}[f_d] = \ConstEfd \, N \left\{ 1 + o(1) \right\} \qquad \text{as $N \to \infty$.}
\end{equation}
For dimensions $d = 1$ and $d = 2$, if the convex hull is not degenerate (i.e., no two points on~$\Sph{1}$
coincide, or, three adjacent points on~$\Sph{2}$ are not on a great circle),
then \MARKED{red}{$f_1 = N$} and $f_2 = 2N-4$. For higher dimensions, the expected
number of facets grows nearly linearly in $N$, but the slope $\ConstEfd$
grows with the dimension:
\begin{center}
\begin{tabular}{c|c|c|c|c|c|c|c|c}
$d$    & $1$ & $2$ & $3$      & $4$       & $5$      & $6$                  & $7$                  & $8$ \\ \hline
$\ConstEfd$ & $1$ & $2$ & $6.7677$ & $31.7778$ & $186.72$ & $1.2964 \times 10^3$ & $1.0262 \times 10^4$ & $9.0425 \times 10^4$
\end{tabular}
\end{center}

Adapting methods for random polytope results (\cite{BuMuTi1985, Mu1990}), 
we derive the large $N$ behavior of the expected value of the sum of $p$th powers of the hole radii
in a random point set on~$\Sph{d}$ for any real $p > 0$, namely that
\begin{equation} \label{eq:expected.hole.radii}
\frac{1}{\mathbb{E}[ f_d ]} \, \mathbb{E}\left[ \sum_{k=1}^{f_d} (\rho_k)^p \right] = \Econstdp \, N^{-p/d} \left\{ 1 + \mathcal{O}( N^{-2/d} ) \right\} \qquad \text{as $N \to \infty$,}
\end{equation}
where $\Econstdp$ is an explicit constant (see Theorem~\ref{thm:sums.of.powers.of.holes.radii}).
For the $2$-sphere, we further obtain next-order terms for such moments (see \eqref{eq:precise.asymptotics}).
The constant $\Econstdp$ in \eqref{eq:expected.hole.radii} can be interpreted as the $p$th moment of a
non-negative random variable $X$ with probability density function
\begin{equation} \label{eq:asymp.density}
\frac{d}{\gammafcn(d)} \left( \kappa_d \right)^d \; x^{d^2-1} \; e^{ - \kappa_d  \, x^d}, \qquad x \geq 0,
\end{equation}
where $\kappa_d$ is defined by \eqref{eq:Sdconst}.
For $d = 1$ the expression \eqref{eq:asymp.density} reduces to the exponential distribution $\frac{1}{\pi} \, e^{- x/\pi}$,
while for $d = 2$ it reduces to the Nakagami distribution (see \cite{Na1960}).
%
%

Based on heuristic arguments and motivated by the numerical experiments in Figure~\ref{fig:hole.radii.empirical}, we also conjecture that as $N \to \infty$ the scaled radii $N^{1/d} \rho_1, \dots, N^{1/d} \rho_{f_d}$ associated with the facets of the convex hull of $N$ random points on $\Sph{d}$ are (dependent) samples from a distribution with a probability density function which converges to \eqref{eq:asymp.density}; see Conjecture~\ref{conj:1}.

Equation \eqref{eq:expected.hole.radii} suggests a conjecture for the expected value \MARKED{red}{$\mathbb{E}[ \rho( \PTCfg_N ) ]$} of the covering radius of $N$ i.i.d.
(independent and identically distributed) random points on $\Sph{d}$. To state this conjecture, we use (here and throughout the paper) the notation $a_N \sim b_N$ as $N \to \infty$ to mean $a_N / b_N \to 1$ as $N \to \infty$. Stated in terms of \MARKED{red}{the} \emph{Euclidean} covering radius, we propose that
\begin{equation} \label{eq:expected.asymptotics.intro}
\mathbb{E}[ \rho( \PTCfg_N ) ] \sim \COVconstd \, \left( \frac{\log N}{N} \right)^{1/d} \qquad \text{as $N \to \infty$}
\end{equation}
and that the coefficient is $\COVconstd = ( \kappa_d )^{-1/d}$ (see Conjecture~\ref{conj:2} and the remark \MARKED{red}{thereafter}).
For $d = 1$, $\COVconstd = \pi$; see \eqref{eq:covering.radius.circle}. 
The conjecture is consistent with results of Maehara \cite{Ma1988} on the probability that equal sized caps cover the sphere: if the radius is larger than the right-hand side of \eqref{eq:expected.asymptotics.intro} by a constant factor, then the probability approaches one as $N\to\infty$, whereas, if the radius is smaller than the right-hand side by a constant factor, then the probability approaches zero.
Observe that the ``mean value'' of the hole radii (obtained by setting $p = 1$ in \eqref{eq:expected.hole.radii}) already achieves the optimal rate of convergence $N^{-1/d}$. In \cite{BoSaRu2012}, Bourgain et al. prove that $\mathbb{E}[ \rho( \PTCfg_N ) ] \leq N^{-1/2 + o(1)}$ for $N$ i.i.d. random points on $\Sph{2}$ and remark that, somewhat surprisingly, the covering radius of random points is much more forgiving compared to their separation properties (cf. \eqref{eq:expected.separation} below).

For the case of the unit circle $\Sph{1}$, order statistics arguments regarding the placement of points and arrangement of ``gaps'' (i.e., the arcs between consecutive points) are described in \cite[p.~133--135, 153]{DaNa2003}. With $\varphi_k$ denoting the arc length of the $k$th largest gap formed by $N$ i.i.d. random points on $\Sph{1}$, one has 
\begin{equation} \label{eq:k.th.gap}
\mathbb{E}[ \varphi_k ] = \frac{2\pi}{N} \left( \frac{1}{N} + \frac{1}{N-1} + \cdots + \frac{1}{k} \right), \qquad k = 1, \dots, N.
\end{equation}
Thus, the \emph{geodesic} covering radius of $N$ random points has the expected value
\begin{equation} \label{eq:covering.radius.circle}
\begin{split}
\mathbb{E}[ \COV( \PTCfg_N ) ]
&= \frac{1}{2} \mathbb{E}[ \varphi_1 ] = \frac{\pi}{N} \left( \frac{1}{N} + \frac{1}{N-1} + \cdots + \frac{1}{2} + 1 \right) \\
&= \pi \frac{\log N}{N} + \frac{\EulerGamma \pi}{N} +  \frac{\pi}{2 \, N^2} + \mathcal{O}(N^{-3}),
\end{split}
\end{equation}
where $\gamma = 0.5772156649\dots$ is the Euler-Mascheroni constant, which implies that \eqref{eq:expected.asymptotics.intro} holds for $d = 1$ and $\COVconstd = \pi$.
For further background concerning related covering processes, see, e.g., \cite{BueCuLo2010,Ma1988,Sh1972}.

Since separation is very sensitive to the placement of points, unsurprisingly, random points have very poor separation properties.
Indeed, \eqref{eq:k.th.gap} yields that the expected value of
the minimal separation of $N$ i.i.d. random points on the unit circle (in the geodesic metric) is 
\begin{equation*}
\mathbb{E}[ \sepDist( \PTCfg_N ) ] = \mathbb{E}[ \varphi_N ] = \frac{2\pi}{N^2}.
\end{equation*}
This is much worse than the minimal separation $2\pi / N$ of $N$ equally spaced points.
Here we deduce a similar result for $\Sph{d}$ (see Corollary~\ref{cor:limit.random.separation}), namely
\begin{equation} \label{eq:expected.separation}
\mathbb{E}[ \sepDist( \PTCfg_N ) ] \sim C_d \, N^{-2/d} \qquad \text{as $N \to \infty$,}
\end{equation}
where $C_d$ is an explicit constant. This rate should be compared with the
optimal separation order $N^{-1/d}$ for best-packing points on
$\Sph{d}$.
%
Using first principles, we further obtain the lower bound
\begin{equation*}
\mathbb{E}[ N^{2/d} \sepDist( \PTCfg_N ) ] \geq C_d \, L_d, \qquad N \geq 2,
\end{equation*}
with an explicit constant $L_d$ and $C_d$ as in \eqref{eq:expected.separation} (see Proposition~\ref{prop:limit.random.separation}).

The outline of our paper is as follows. In Section~\ref{sec:covering.random.points}, we state our results concerning the moments of the hole radii \MARKED{red}{as well as conjectures dealing with the distribution of these radii and the asymptotic covering radius of random points on $\Sph{d}$. Also included there are graphical representations of numerical data supporting these conjectures.}
%
Section~\ref{sec:separation.random.points} is devoted \MARKED{red}{to the statements of} separation results both in terms of probability and expectation. 
%
In Section~\ref{sec:proofs}, we collect proofs. In Section~\ref{sec:comparison}, we briefly compare numerically the separation and hole radius properties of pseudo-random point sets with those of several popular non-random point configurations on $\Sph{2}$.

\section{Covering of Random Points on the Sphere}
\label{sec:covering.random.points}

\subsection{Expected moments of hole radii}

The facets of the convex hull of an $N$-point set on $\Sph{d}$ are in
one-to-one correspondence with a family of maximal ``spherical cap
shaped'' holes. Each facet determines a $d$-dimensional hyperplane that
divides $\Sph{d}$ into an open spherical cap containing no points and
a complementary closed spherical cap including all points of the
configuration, with at least $d+1$ of them on the boundary.
The convex hull $K_N$ of $N$ independent uniformly distributed random points $\PT{X}_1, \dots, \PT{X}_N$ on $\Sph{d}$ is a \emph{random polytope} with vertices on~$\Sph{d}$ and $f_d$ facets. These facets determine the geodesic radii $\alpha_1, \dots, \alpha_{f_d} \in [0, \pi]$ of all the $f_d$ holes in the configuration.

This connection to convex random polytopes can be exploited to derive probabilistic assertions for the sizes of the holes in a random point set generated by a positive probability density function.
Indeed, the proof of our first result uses geometric considerations and general results on the approximation of convex sets by random polytopes with vertices on the boundary (cf. \cite{BoKaFeHu2013,Rei2002}) to establish large $N$ asymptotics for the expected value of a ``natural'' weighted sum of the (squared) Euclidean radii ${\rho_1 = 2 \sin( \alpha_1 / 2 )}, \dots, {\rho_{f_d} = 2 \sin( \alpha_{f_d} / 2 )}$.

\begin{prop} \label{prop:holes.radii}
Let $\mu$ be a probability measure absolutely continuous with respect to the normalized surface area measure $\sigma_d$ on $\Sph{d}$ with positive continuous density $\eta$. If $\PT{Y}_1, \dots, \PT{Y}_N$ are $N$ points on $\Sph{d}$ that are randomly and independently distributed with respect to $\mu$, then
\begin{equation}
\begin{split} \label{eq:expected.weighted.sum}
\mathbb{E}_{\mu}\Bigg[ \sum_{k=1}^{f_d} \frac{A_{k,N}}{A_N} \, (\rho_k)^2 \Bigg]
&\sim N^{-2/d} \, \frac{\gammafcn(d + 1 + 2/d)}{\gammafcn( d + 1 )} \left( \kappa_d \right)^{-2/d} \int_{\Sph{d}} \left[ \eta( \PT{x} ) \right]^{-2/d} \dd \sigma_d( \PT{x} ) \\
&\phantom{=}+ 2 \, \mathbb{E}_{\mu}\Bigg[ \left( 1 - \frac{A_N}{A} \right) \left( 1 - \frac{V_N / V}{A_N / A} \right) \Bigg] \qquad \text{as $N \to \infty$,}
\end{split}
\end{equation}
where $A_{k,N}$ is the surface area of the $k$th facet of the convex hull, $A_N = \sum_{k=1}^{f_d} A_{k,N}$, and $V_N$ is the volume of the convex hull.
In the last term, $A$ and $V$ denote the surface area of~$\Sph{d}$ and the volume of the unit ball in $\R{d+1}$, respectively.
\end{prop}

By $\mathbb{E}_{\mu}$ we mean the expected value with respect to $\mu$. The proof of Proposition~\ref{prop:holes.radii} and other results are given in Section~\ref{sec:proofs}.

We now state one of our main results, which deals with i.i.d. uniformly chosen random points on $\Sph{d}$.

\begin{thm} \label{thm:sums.of.powers.of.holes.radii}
If $p \geq 0$ and $\PT{X}_1, \dots, \PT{X}_N$ are $N$ points on $\Sph{d}$ that are independently and randomly distributed with respect to $\sigma_d$, then
\begin{align}
\mathbb{E}\left[ \sum_{k=1}^{f_d} (\rho_k)^p \right]
&= \ConstEfd \left( \kappa_d \right)^{-p/d} \frac{\gammafcn( d + p/d ) \gammafcn( N + 1 )}{\gammafcn(d) \gammafcn( N + p/d )} \left\{ 1 + \mathcal{O}( N^{-2/d} ) \right\} \label{eq:sums.of.powers.asymptotics} \\
&= c_{d,p} \; N^{1-p/d} \left\{ 1 + \mathcal{O}( N^{-2/d} ) \right\} \notag
\end{align}
as $N \to \infty$, where $\rho_k = \rho_{k,N}$ is the Euclidean hole radius associated with the $k$th facet of the convex hull of $\PT{X}_1, \dots, \PT{X}_N$,
\begin{equation} \label{eq:cdp}
c_{d,p} \DEF \ConstEfd \; \Econstdp, \qquad \Econstdp \DEF \frac{\gammafcn( d + p/d )}{\gammafcn(d)} \left( \kappa_d \right)^{-p/d},
\end{equation}
and $\kappa_d$, $\ConstEfd$ are defined in (\ref{eq:Sdconst}). \MARKED{red}{The $\mathcal{O}$-terms in \eqref{eq:sums.of.powers.asymptotics} depend only on $d$ and $p$.}
\end{thm}

\MARKED{red}{
Figures~\ref{fig:S2HolRadp} and \ref{fig:SdHolRadd} illustrate empirical data that are in good agreement with the assertions of Theorem~\ref{thm:sums.of.powers.of.holes.radii}.} 

\begin{figure}[ht]
\begin{center}
\includegraphics[scale=0.70]{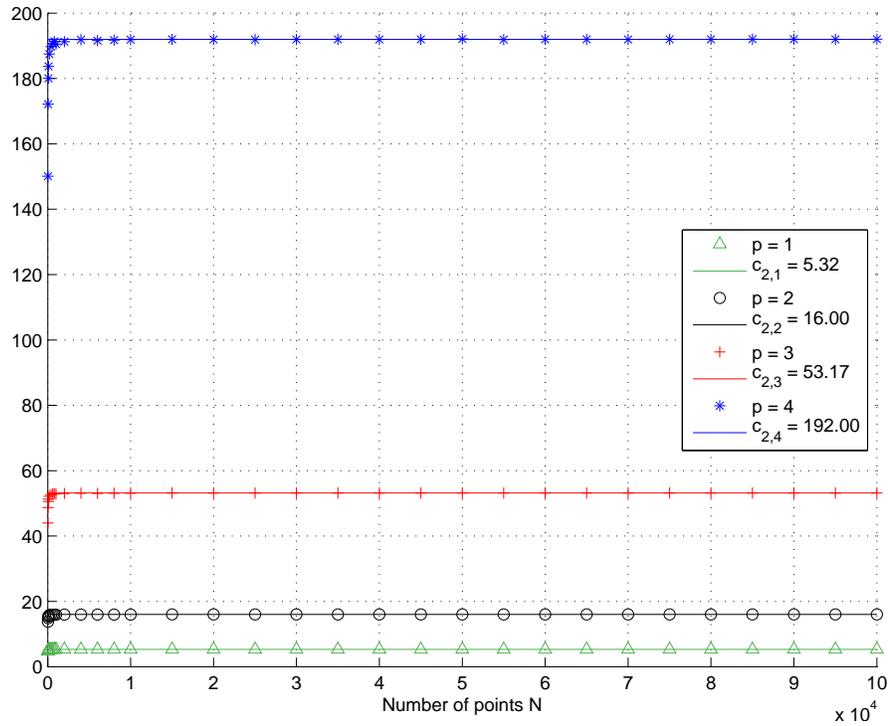} \quad
\caption{Expected value \eqref{eq:sums.of.powers.asymptotics}
of the sum of Euclidean hole radii on $\Sph{2}$ to power $p = 1, 2, 3, 4$,
estimated using $400$ samples of $N$ uniformly distributed points and
scaled by dividing the sample mean by $N^{1-p/2}$,
compared with $c_{2,p}$ given by \eqref{eq:cdp}}
\label{fig:S2HolRadp}
\end{center}
\end{figure}
\begin{figure}[ht]
\begin{center}
\includegraphics[scale=0.70]{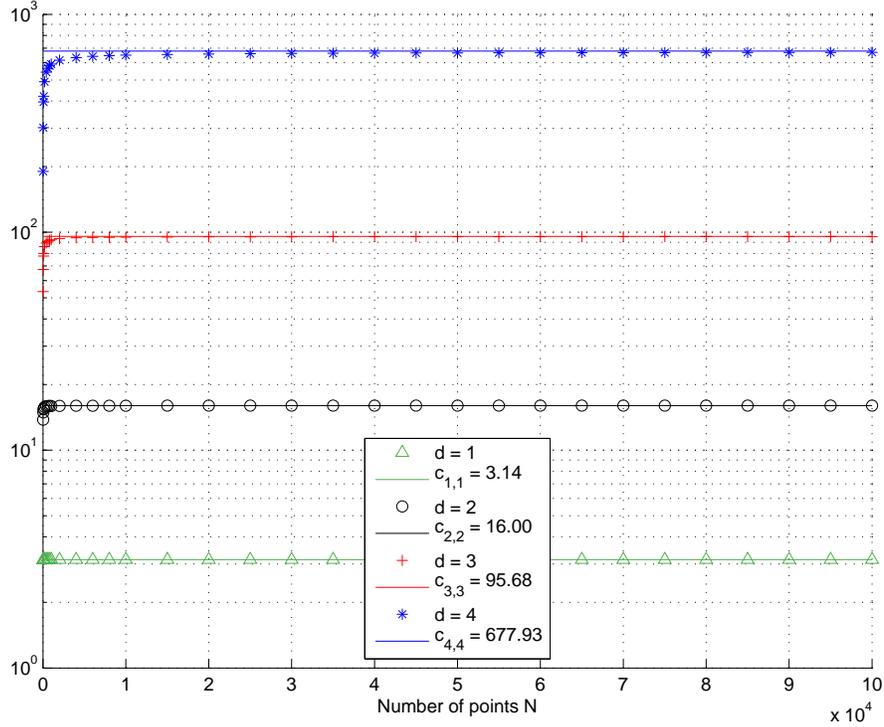} \quad
\caption{
Expected value \eqref{eq:sums.of.powers.asymptotics}
of the sum of Euclidean hole radii on $\Sph{d}$ to power $p = d$,
estimated using $400$ samples of $N$ uniformly distributed points
compared with $c_{d,d}$ given in \eqref{eq:cdp},
for $d = 1, 2, 3, 4$}
\label{fig:SdHolRadd}
\end{center}
\end{figure}

\MARKED{red}{
Notice that for $p=0$ we deduce the following result for the expected number of facets,
\begin{equation} \label{eq:f.d.estimate}
\mathbb{E}\left[ f_d \right] = B_d \; N \left\{ 1 + \mathcal{O}( N^{-2/d} ) \right\} \qquad \text{as $N \to \infty$,}
\end{equation}
which again confirms \eqref{eq:Buchta.et.al}. 
}

\subsubsection{More precise estimates for $\Sph{2}$}
In the case of $\Sph{2}$ we have the more precise asymptotic form (cf. proof of Theorem~\ref{thm:sums.of.powers.of.holes.radii})
\begin{equation} \label{eq:precise.asymptotics}
\begin{split}
&\mathbb{E}\left[ \sum_{k = 1}^{f_2} \left( \rho_k \right)^p \right]
= \left( 2 N - 4 \right) 2^{p} \frac{\gammafcn( 2 + \frac{p}{2} ) \gammafcn( N + 1 )}{\gammafcn( N + 1 + p/2 )} \left\{ 1 + o(1) \right\} \\ 
&\phantom{equals}= \left( 2 N - 4 \right) 2^{p} \gammafcn( 2 + \frac{p}{2} ) \, N^{-p/2} \Bigg\{ 1 + \sum_{\ell=1}^{L-1} \frac{\binom{-p/2}{\ell} \, B_\ell^{(1-p/2)}(1)}{N^\ell} + \mathcal{O}\left( \frac{1}{N^L} \right) \Bigg\},
\end{split}
\end{equation}
where we have omitted an error term that goes exponentially fast to zero as $N \to \infty$.
Here $B_{\ell}^{(a)}(x)$ denotes the generalized Bernoulli polynomial with $B_1^{(1-p/2)}(1) = (1 + p/2) / 2$.
Using \eqref{eq:precise.asymptotics}, we get the following improvement of \eqref{eq:f.d.estimate}: for any positive integer $L$, 
\begin{equation*}
\mathbb{E}[ f_2 ] = \left( 2 N - 4 \right) \Bigg\{ 1 + \mathcal{O}\left( \frac{1}{N^L} \right) \Bigg\} \qquad \text{as $N \to \infty$.}
\end{equation*}
Recall that Euler's celebrated Polyhedral Formula states 
that the number of vertices $f_0$, edges $f_1$ and faces $f_2$ of a convex
polytope satisfy $f_0 - f_1 + f_2 = 2$ and Steinitz~\cite{St1906} proved
that the conditions
\begin{equation*}
f_0 - f_1 + f_2 = 2, \qquad f_2 \leq 2 \, f_0 - 4, \qquad f_0 \leq 2 f_2 - 4
\end{equation*}
are necessary and sufficient for $(f_0,f_1,f_2)$ to be the associated triple for the convex polytope~$K_N$. In particular, if $K_N$ is \emph{simplicial} (all faces are simplices), then $f_2 = 2 \, f_0 - 4$.

The sum of all hole radii, on average, behaves for large $N$ like
\begin{equation*}
\mathbb{E}\left[ \sum_{k = 1}^{f_2} \rho_k \right] \sim \left( 2 N - 4 \right) \left( 2 N \right) \frac{(N-1)!}{\Pochhsymb{5/2}{N-1}} = \left( 2 N - 4 \right) \frac{3}{2} \sqrt{\pi} \, N^{-1/2} \left\{ 1 - \frac{3}{8} \, \frac{1}{N} + \mathcal{O}( \frac{1}{N^2} ) \right\}.
\end{equation*}
Here, $\Pochhsymb{a}{n}$ denotes the Pochhammer symbol defined by $\Pochhsymb{a}{0} \DEF 1$ and $\Pochhsymb{a}{n+1} \DEF ( n + a ) \Pochhsymb{a}{n}$.

Furthermore, on observing that the $\sigma_2$-surface area measure of a spherical cap $C_{\rho_k}$ with Euclidean radius $\rho_k$ is given by (remembering that $\alpha_k$ is the geodesic radius of each cap)
\begin{equation*}
\sigma_2( C_{\rho_k} ) = \frac{\omega_{1}}{\omega_2} \int_{\cos \alpha_k}^1 \dd t = \frac{1}{4} \left( 2 - 2 \; \cos \alpha_k \right) = \frac{\rho_k^2}{4},
\end{equation*}
we obtain
\begin{equation*}
\mathbb{E}\left[ \sum_{k=1}^{f_2} \sigma_2( C_{\rho_k} ) \right] 
= \frac{1}{4} \mathbb{E}\left[ \sum_{k=1}^{f_2} \rho_k^2 \right] 
\sim \frac{2\left( 2 N - 4 \right)}{N+1} 
\sim  4 - \frac{12}{N} \qquad \text{as $N \to \infty$.}
\end{equation*}

The analogous formula for higher-dimensional spheres is
\begin{equation*}
\mathbb{E}\left[ \sum_{k=1}^{f_d} \sigma_d( C_{\rho_k} ) \right] \sim d \, \ConstEfd \left\{ 1 + \mathcal{O}( N^{-2/d} ) \right\} \qquad \text{as $N \to \infty$,}
\end{equation*}
where $\ConstEfd$ is given in \eqref{eq:Sdconst}. (This follows from an inspection of the proof of Theorem~\ref{thm:sums.of.powers.of.holes.radii}.)



\subsection{Heuristics leading to Conjectures~\ref{conj:1} and \ref{conj:2}} The analysis for Theorem~\ref{thm:sums.of.powers.of.holes.radii} relies on the following asymptotic approximation of the expected value
\begin{equation*}
\mathbb{E}\left[ \sum_{k = 1}^{f_d} \left( \rho_k \right)^p \right] \sim \frac{2}{\betafcn( d^2/2, 1/2 )} \, \binom{N}{d+1} \, \int_0^{\tau} \rho^p \left[ 1 - \sigma_d( C_{\rho} ) \right]^{N-d-1} \, \rho^{d^2-1} \left( 1 - \frac{\rho^2}{4} \right)^{d^2/2 - 1} \dd \rho,
\end{equation*}
where $\betafcn(a,b)$ is the Beta function (see \cite[Section~5.12]{NIST:DLMF})
defined by
\begin{equation*}
\MARKED{red}{\betafcn( a, b )} \DEF \int_0^1 t^{a-1}(1-t)^{b-1}\dd t = \frac{\Gamma(a)\Gamma(b)}{\Gamma(a+b)},
\end{equation*}
and,
as before, $\sigma_d( C_{\rho} )$ is the $\sigma_d$-surface area of a spherical cap of Euclidean radius $\rho$. The leading term in the asymptotic approximation (as $N \to \infty$) is not affected by the choice of $\tau \in (0, 2)$. (A change in $\tau$ yields a change in a remainder term not shown here that goes exponentially fast to zero for fixed (or sufficiently weakly growing) $p$.) Apart from a normalization factor (essentially $\mathbb{E}[ f_d ]$), this right-hand side can be interpreted as an approximation of the $p$th moment of the random variable ``hole radius'' associated with a facet of the convex hull of the~$N$ i.i.d. uniformly distributed random points on $\Sph{d}$ whereas the left-hand side up to the normalization can be seen as the empirical distribution of the identically (but not independently) distributed hole radii.
The normalized surface area of $C_{\rho}$ can be expressed in terms of a regularized incomplete beta function, or equivalently, a Gauss hypergeometric function,
\begin{equation} \label{eq:sigma_d.C.rho}
\sigma_d( C_{\rho} ) = \IncompleteBetaRegularized_{\rho^2/4}\left(\frac{d}{2}, \frac{d}{2} \right) = \kappa_d \, \rho^d \, \Hypergeom{2}{1}{1-d/2,d/2}{\MARKED{red}{1+}d/2}{\frac{\rho^2}{4}},
\end{equation}
where (see \cite[Eq. 8.17.1]{NIST:DLMF})
\begin{equation*}
\IncompleteBetaRegularized_x(a,b) \DEF \frac{\int_0^x t^{a-1}(1-t)^{b-1}\dd t}{\int_0^1 t^{a-1}(1-t)^{b-1}\dd t}=\frac{\Gamma(a+b)}{\Gamma(a)\Gamma(b)}\int_0^x t^{a-1}(1-t)^{b-1}\dd t.
\end{equation*}
Hence the change of variable $x = \rho \, N^{1/d}$ leads in
a natural way to the substitution
\begin{equation*}
  \left[ 1 - \sigma_d( C_{x / N^{1/d}} ) \right]^{N-d-1} \approx \left( 1 - \kappa_d \, \frac{x^d}{N} \right)^{N-d-1} \to
  \exp\Big( - \kappa_d \, x^d \Big) \qquad \text{as $N \to \infty$}
\end{equation*}
and thus by the dominated convergence theorem to the
limit relation
\begin{equation} \label{eq:hole.radius.limit.distribution}
\frac{1}{\mathbb{E}[f_d]} \, \mathbb{E}\left[ \sum_{k = 1}^{f_d} \left( \rho_k \, N^{1/d} \right)^p \right] \to
  \int_0^\infty x^p \, e^{ - \kappa_d \, x^d} \, x^{d^2-1} \, \dd x \Bigg/ \int_0^\infty e^{ - \kappa_d \, x^d } \, x^{d^2-1} \, \dd x 
\end{equation}
as $N \to \infty$.
Either integral can be easily computed using the well-known integral representation of the gamma function
(\cite[Eq.~5.9.1]{NIST:DLMF}, \cite{Olver:2010:NHMF}),
\begin{equation*}
  \frac{1}{\mu} \gammafcn( \tfrac{\nu}{\mu} ) \, \frac{1}{z^{\nu / \mu}} = \int_0^\infty e^{- z \, t^\mu} \, t^{\nu - 1} \, \dd t.
\end{equation*}
In fact, the right-hand side of \eqref{eq:hole.radius.limit.distribution} reduces to the same constant
\begin{equation*}
\frac{\gammafcn( d + p/d )}{\gammafcn(d)} \left( \kappa_d \right)^{-p/d}
\end{equation*}
we obtained in Theorem~\ref{thm:sums.of.powers.of.holes.radii}.

Relation \eqref{eq:hole.radius.limit.distribution} suggests that the scaled hole radii $x_1 = \rho_1 N^{1/d}$, \dots, $x_{f_d}
= \rho_{f_d} N^{1/d}$ associated with the facets of the convex hull of
the~$N$ i.i.d. uniformly distributed random points on $\Sph{d}$ are
identically (but not independently) distributed with respect to a PDF that
tends (as $N \to \infty$) to the limit PDF \eqref{eq:hole.radius.PDF}.

\begin{conj} \label{conj:1}
The scaled hole radii $N^{1/d} \rho_1, \dots, N^{1/d} \rho_{f_d}$ associated with the facets of the convex hull of
$N$ i.i.d. random points on $\Sph{d}$ distributed with respect to $\sigma_d$ are (dependent) samples from a distribution with probability density function which converges, as $N \to\infty$, to the limiting distribution with PDF
\begin{equation}\label{eq:hole.radius.PDF}
  f(x) = \frac{d}{\Gamma(d)}\; \kappa_d^d \; x^{d^2-1} \; e^{-\kappa_d x^d}, \qquad x \geq 0,
\end{equation}
where $\kappa_d$ is defined in \eqref{eq:Sdconst}.
\end{conj}

\begin{rmk}
Note that the transformation $Y = X^d$ gives the PDF
\begin{equation*}
\frac{1}{\Gamma(k) \theta^k} \; y^{k-1} \; e^{-y/\theta}, \qquad y \geq 0,
\end{equation*}
where $k = d$ and $\theta = 1/\kappa_d$, showing that $Y = X^d = N \rho^d$ has a Gamma distribution. The surface area of a spherical cap with small radius $\rho$ is approximately given by a constant multiple of $\rho^d$. Thus the surface area of the caps associated with each hole follows a Gamma distribution.
\end{rmk}

Figure~\ref{fig:hole.radii.empirical} illustrates the empirical distribution of the scaled hole radii of i.i.d. uniformly distributed random points on $\Sph{d}$ for $d = 1,2,3,4$ which is in good agreement with the conjectured limiting PDF \eqref{eq:hole.radius.PDF}. For larger values $d > 4$, the convergence of the empirical distribution to the limiting distribution is slow.

\begin{figure}[ht]
\begin{center}
\includegraphics[scale=0.425]{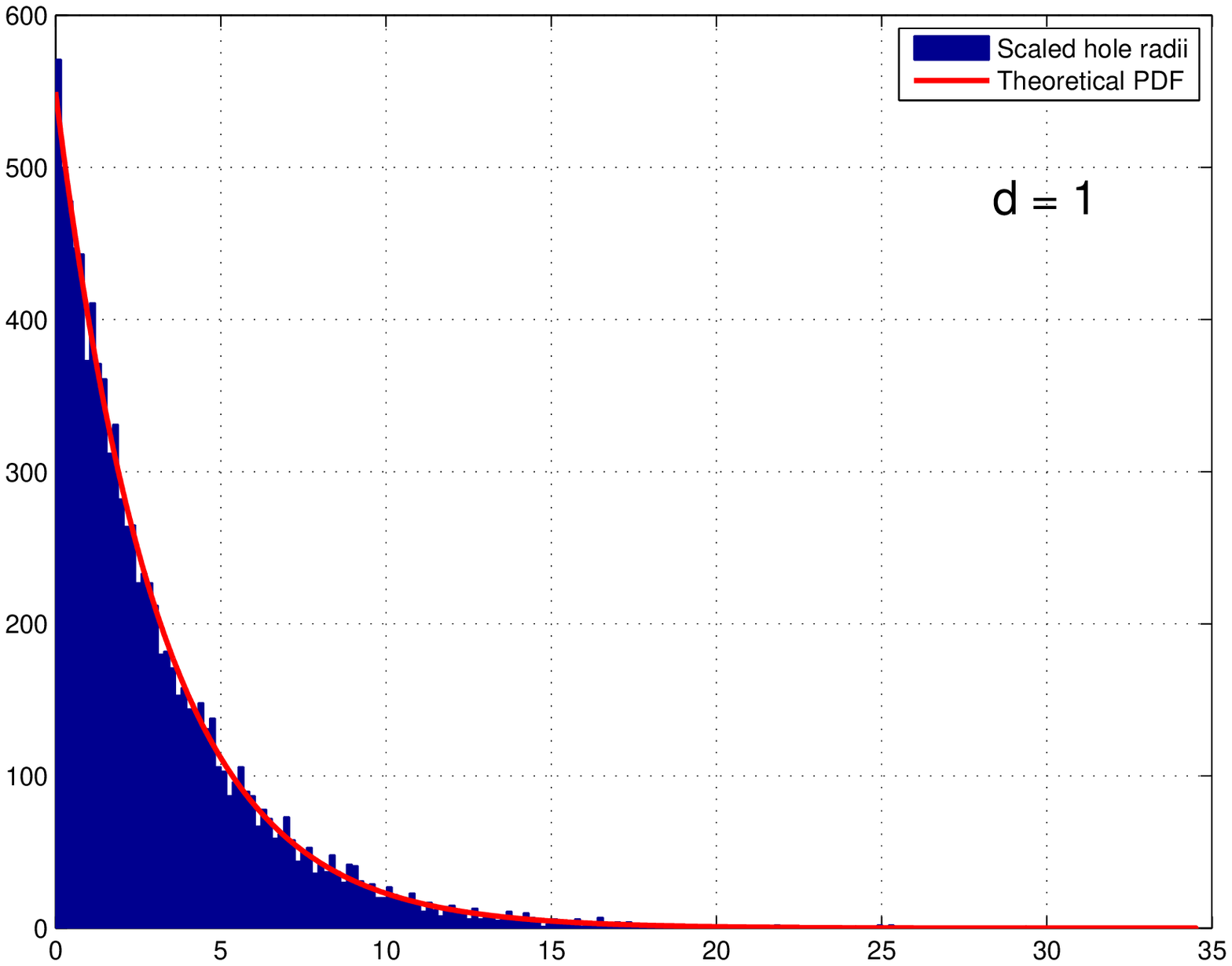} \quad
\includegraphics[scale=0.425]{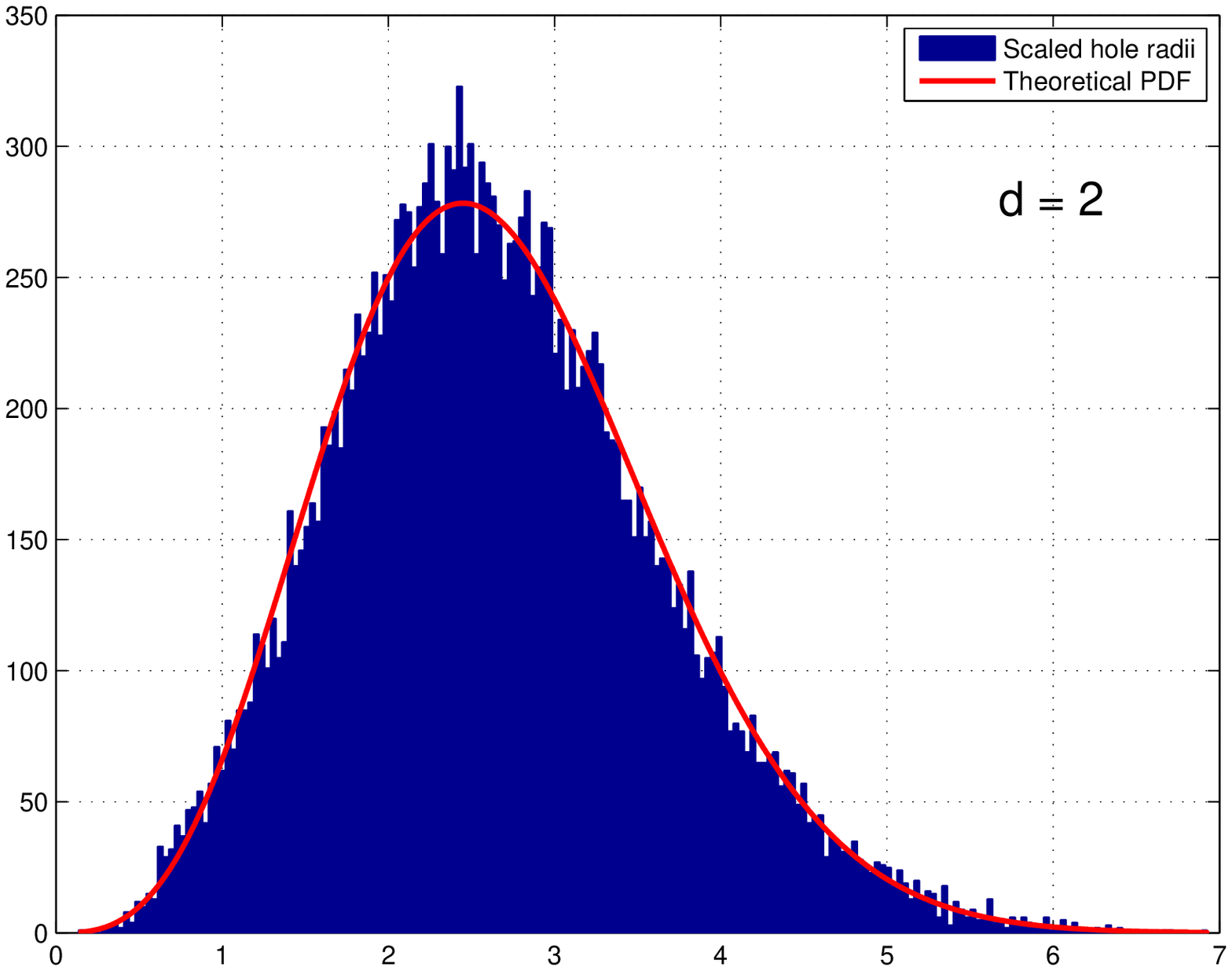} \\
\includegraphics[scale=0.425]{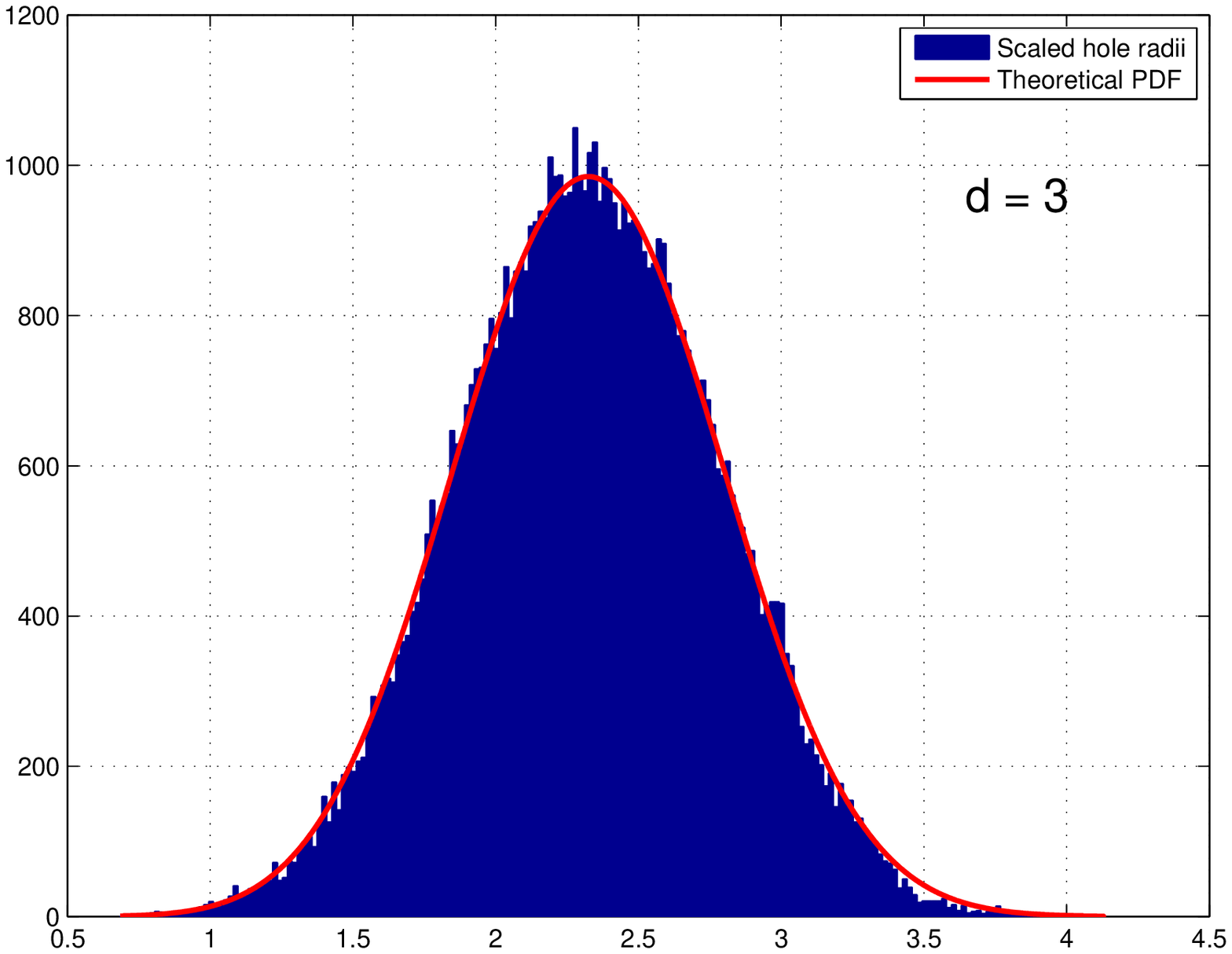} \quad
\includegraphics[scale=0.425]{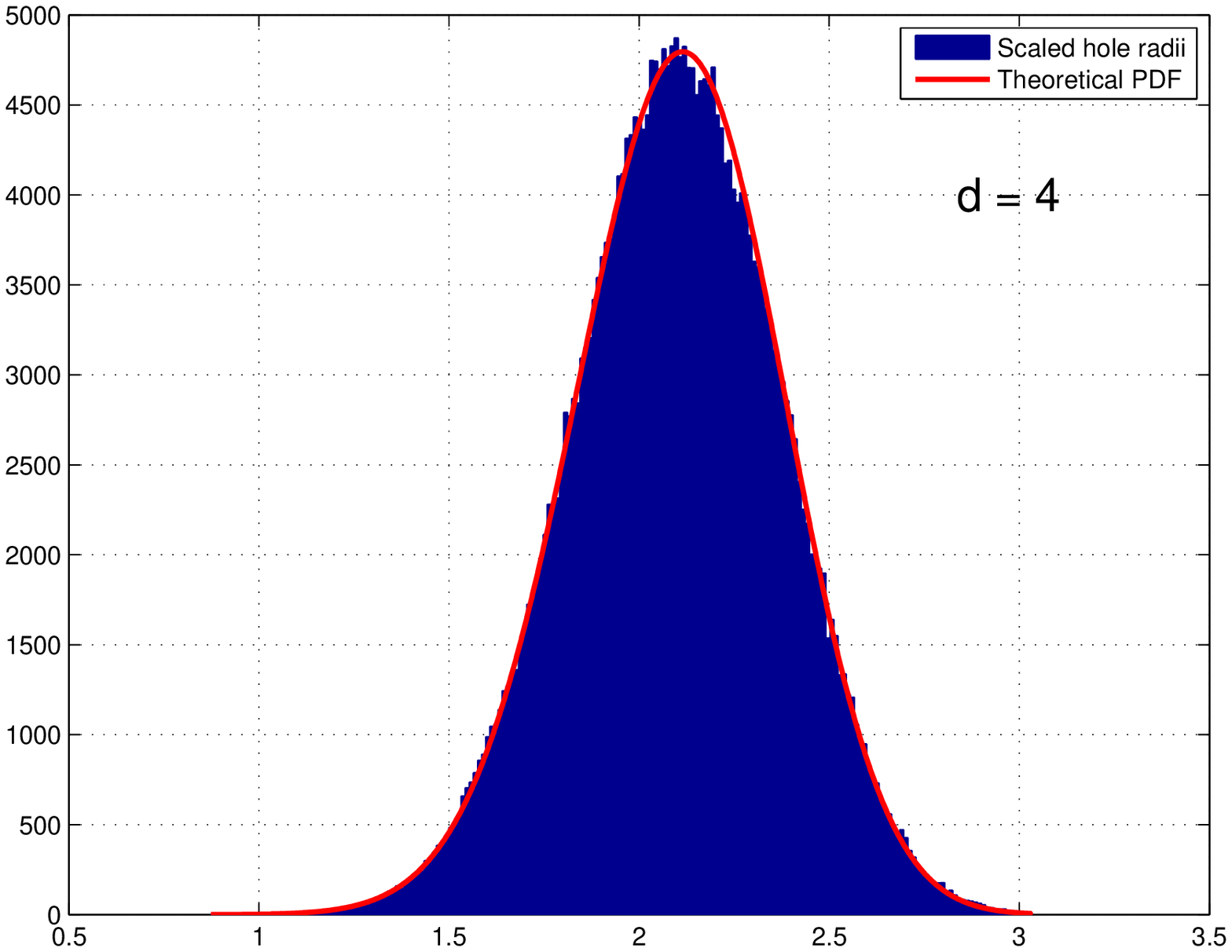} \quad
\caption{Histograms of the scaled Euclidean hole radii $N^{1/d} \rho_k$, $k = 1,\ldots,f_d$
for a sample of $N = 10,000$ uniformly distributed points on $\Sph{d}$
for $d = 1, 2, 3, 4$, and the conjectured limiting PDF \eqref{eq:hole.radius.PDF}}
\label{fig:hole.radii.empirical}
\end{center}
\end{figure}
Figure~\ref{fig:hole.radi.limit.distribution} illustrates the conjectured limiting PDF \eqref{eq:hole.radius.PDF}, showing how the distribution changes with varying dimension $d$.
\begin{figure}[htbp]
\begin{center}
\includegraphics[scale=0.70]{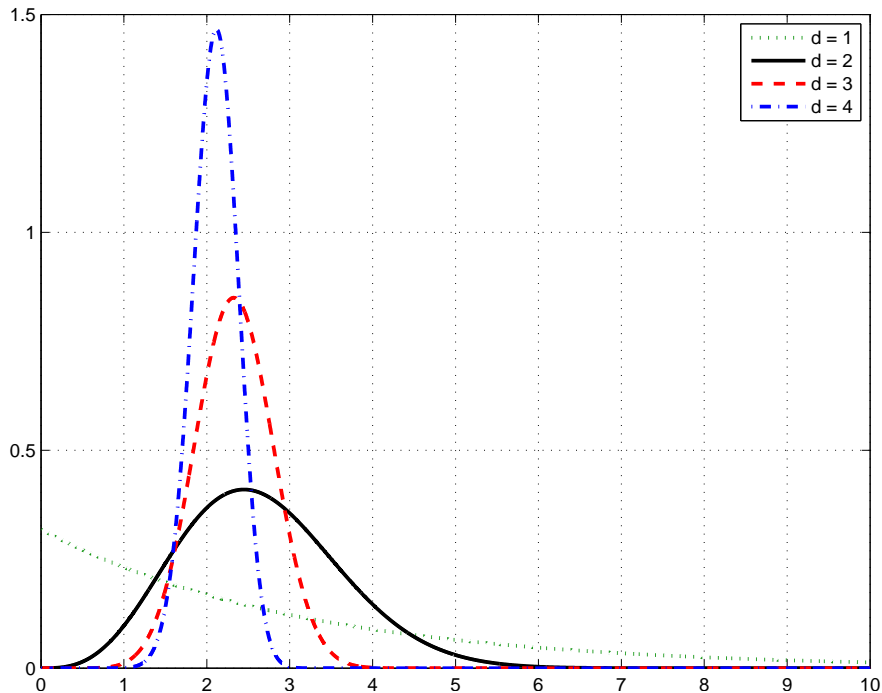} \quad
\caption{\label{fig:hole.radi.limit.distribution} Conjectured limiting distribution \eqref{eq:hole.radius.PDF} on $[0, 10]$ for dimensions $d = 1, 2, 3, 4$}
\end{center}
\end{figure}
The CDF corresponding to \eqref{eq:hole.radius.PDF} can be expressed in terms of the regularized incomplete gamma function
\begin{equation*}
\gammafcnregularizedP( a, x ) \DEF \frac{\gamma(a, x)}{\gammafcn( a )}, \qquad \gamma( a, x ) \DEF \int_0^x e^{-t} t^{a-1} \dd t
\end{equation*}
by means of
\begin{equation*}
\int_0^x \frac{d}{\gammafcn(d)} \left( \kappa_d \right)^d  e^{- \kappa_d \, t^d} \, t^{d^2-1} \dd t = \frac{1}{\gammafcn(d)} \int_0^{ \kappa_d \, x^d} e^{-u} \, u^{d-1} \dd u = \gammafcnregularizedP\Big( d, \kappa_d \, x^d \Big).
\end{equation*}
For $d = 2$ the PDF in \eqref{eq:hole.radius.PDF} reduces to the \emph{Nakagami distribution} (see, e.g., \cite{KoJiJa2004})
\begin{equation*}
\frac{2 \nu^\nu}{\gammafcn( \nu ) \, \Omega^\nu} \, x^{2\nu - 1} \, e^{- \frac{\nu}{\Omega} \, x^2}, \qquad x \geq 0,
\end{equation*}
with shape parameter $\nu = 2$ and spread $\Omega = 8$.

\subsection*{The covering radius of i.i.d. uniformly distributed points on $\Sph{d}$}

By Conjecture~\ref{conj:1}, the unordered scaled hole radii
of $N$ i.i.d. uniformly chosen random points on $\Sph{d}$ are
identically (but not independently) distributed with respect to a PDF
tending to \eqref{eq:hole.radius.PDF}.
We now order the $f_d$ hole radii such that $\MARKED{red}{\rho = }\rho_1 \geq \rho_2 \geq \cdots \geq
\rho_{f_d}$ and ask about the distribution properties of the $k$th largest
hole radius. Of particular interest is the largest hole radius~$\rho_1$
which gives the (Euclidean) covering radius of the random point
configuration.
%
%
\begin{conj} \label{conj:2}
For $d \geq 2$ and $\COVconstd = ( \kappa_d )^{-1/d}$,
the Euclidean covering radius of $N$ i.i.d. uniformly distributed points on $\Sph{d}$ satisfies
\begin{equation} \label{eq:conj.E.rho}
  \mathbb{E}[ \rho(\PTCfg_N) ] \sim \COVconstd
  \left( \frac{\log N}{N} \right)^{1/d} \qquad \text{as $N \to \infty$.}
\end{equation}
\end{conj}

\begin{rmk}
Inspired by this conjecture, Reznikov and \MARKED{red}{Saff~\cite{RezS2015online}} recently investigated this problem and proved results for more general manifolds. Note that Theorem~\ref{thm:sums.of.powers.of.holes.radii} implies that
\begin{equation*}
\frac{1}{\mathbb{E}[f_d]} \, \mathbb{E}\left[ \sum_{k=1}^{f_d} \rho_k \right] \sim \Econstdone \, N^{-1/d} \left\{ 1 + \mathcal{O}( N^{-2/d} ) \right\} \qquad \text{as $N \to \infty$},
\end{equation*}
and thus the convergence rate of this ``mean value'' of the hole radii is smaller by a factor $( \log N )^{1/d}$ than the conjectured rate for the expected value of the maximum of the hole radii in \eqref{eq:conj.E.rho}.
\end{rmk}

\begin{proof}[Justification of Conjecture~\ref{conj:2}]
Dividing the asymptotics in equation \eqref{eq:main.E.sum.relation} of Section~\ref{sec:proofs} through by the asymptotics for $\mathbb{E}[f_d]$ (given by \eqref{eq:main.E.sum.relation} for $p=0$) and taking the $p$th root, we arrive at
\begin{equation} \label{eq:crucial.asymptotics}
\left( \frac{1}{\mathbb{E}[f_d]} \, \mathbb{E}\left[ \sum_{k=1}^{f_d} (\rho_k)^p \right] \right)^{1/p} \sim
\left( \kappa_d \right)^{-1/d} \left[ \frac{\gammafcn( d + p/d ) \gammafcn( N )}{\gammafcn( N + p/d )} \right]^{1/p} \left\{ 1 + \mathcal{R}(N,p) \right\}^{1/p}
\end{equation}
as $N \to \infty$. Not shown here is a remainder term that goes exponentially fast to $0$ for fixed (or sufficiently weakly growing) $p$.
A qualitative analysis of the right-hand side of~\eqref{eq:crucial.asymptotics} for $p$ growing like $p = \eps \, d \, \log N$ for some $\eps > 0$ as $N \to \infty$ provides the basis for our conjecture. It is assumed that $p$ grows sufficiently slowly so that the right-hand side of \eqref{eq:crucial.asymptotics} provides the leading term of the asymptotics.
From the proof of Theorem~\ref{thm:sums.of.powers.of.holes.radii} we see that $\mathcal{R}(N,p) \leq C \, N^{-2\xi / d}$ for some $\xi \in (0,1)$.
Thus, the expression $\{ 1 + \mathcal{R}(N,p) \}^{1/p}$ goes to $1$ as $N \to \infty$. The well-known asymptotic expansion of the log-gamma function $\log \gammafcn( x )$ for large $x$ yields 
\begin{equation*}
\begin{split}
\left[ \frac{\gammafcn( d + p/d ) \gammafcn( N )}{\gammafcn( N + p/d )} \right]^{1/p}
= \exp\Bigg[ &\frac{1}{d} \log \log N - \frac{1}{d} \log N + \frac{1}{d} \log \eps \\
&\phantom{=}- \frac{N-1/2}{N} \, \frac{\log\big( 1 + \eps (\log N) / N \big)}{d \, \eps (\log N) / N } + \mathcal{O}\big( \frac{\log \log N}{\log N} \big) \Bigg]
\end{split}
\end{equation*}
as $N \to \infty$. Hence,
\begin{equation*}
\left( \frac{1}{\mathbb{E}[f_d]} \, \mathbb{E}\left[ \sum_{k=1}^{f_d} (\rho_k)^p \right] \right)^{1/p} \sim \left( \kappa_d \right)^{-1/d} \left( \frac{\log N}{N} \right)^{1/d} e^{\frac{1}{d} \log \eps - \frac{1}{d} + \mathcal{O}( \frac{\log \log N}{\log N} )} + \mathcal{R}(N)
\end{equation*}
as $N \to \infty$.
Maehara's results in \cite{Ma1988} then give that $\COVconstd = ( \kappa_d )^{-1/d}$.
\end{proof}

We remark that further numerical experiments support Conjecture~\ref{conj:2}. They also show that the convergence is very slow when the dimension $d$ gets larger.

\section{Separation of Random Points}
\label{sec:separation.random.points}

Let $\PT{X}_1, \dots, \PT{X}_N$ form a collection of $N\geq 2$ independent uniformly distributed random points on $\Sph{d}$.
These $N$ random points determine $\binom{N}{2}$ random angles $\Theta_{j,k} \in [0, \pi]$ via $\cos \Theta_{j,k} = \PT{X}_j \cdot \PT{X}_k$ ($1 \leq j < k \leq N$). In the case of fixed dimension $d$ \footnote{The case when both $d$ and $N$ grow is also discussed in \cite{CaFaJi2013}.}, the global behavior of these pairwise angles is captured by their empirical distribution
\begin{equation*}
\mu_N \DEF \mu[\PT{X}_1, \dots, \PT{X}_N] \DEF \frac{1}{\binom{N}{2}} \mathop{\sum_{j=1}^N \sum_{k=1}^N}_{j < k} \delta_{\Theta_{j,k}}, \qquad N \geq 2,
\end{equation*}
where $\delta_\theta$ is the Dirac-delta measure with unit mass at $\theta \in [0, \pi]$. The empirical law of such angles among a large number of random points on $\Sph{d}$ is known. 

\MARKED{red}{\begin{thm}[{\cite[Theorem~1]{CaFaJi2013}}] \label{thm:CaFaJi.1}
With probability one, the empirical distribution of random angles $\mu_N$ converges weakly (as $N \to \infty$) to the limiting probability density function
\begin{equation*}
h( \theta ) \DEF \frac{\omega_{d-1}}{\omega_d} \left( \sin \theta \right)^{d-1}, \qquad \theta \in [0, \pi].
\end{equation*}
\end{thm}
}

\begin{rmk}
All angles $\Theta_{j,k}$ ($1 \leq j < k \leq N$) are identically distributed with respect to the probability density function $h( \theta )$ (\cite[Lemma~6.2(i)]{CaFaJi2013}) but are not independent (some are large and some are small).
\end{rmk}

\begin{rmk}
The PDF $h( \theta )$ appears in the following decomposition of the normalized surface area measure $\sigma_d$ on $\Sph{d}$ (cf. \cite{Mu1966})
\begin{equation*}
  \dd \sigma_d( \PT{x} ) =
  \frac{\omega_{d-1}}{\omega_d} \left( \sin \theta \right)^{d-1} \dd \theta \, \dd \sigma_{d-1}( \overline{\PT{x}} ), \qquad
  \PT{x} = ( \sin \theta \, \overline{\PT{x}}, \cos \theta ), \, \theta \in [0, \pi ], \overline{\PT{x}} \in \Sph{d-1}.
\end{equation*}
Thus, the cumulative distribution function associated with $h( \theta )$,
\begin{equation*}
H( \theta ) = \int_0^\theta h( \theta^\prime ) \dd \theta^\prime = \frac{\omega_{d-1}}{\omega_d} \int_0^\theta \int_{\Sph{d-1}} \left( \sin \theta \right)^{d-1} \dd \theta \, \dd \sigma_{d-1}( \overline{\PT{x}} ) = \sigma_d( \{ \PT{x} \in \Sph{d} : \PT{x} \cdot \PT{p} \geq \cos \theta \} ),
\end{equation*}
measures the (normalized) surface area of a spherical cap with arbitrary center $\PT{p}\in\Sph{d}$ and geodesic radius $\theta \in [0, \pi]$.
\end{rmk}

On a high-dimensional sphere most of the angles are concentrated around $\pi / 2$. The following concentration result provides a precise characterization of the folklore that ``all high-dimensional unit random vectors are almost always nearly orthogonal to each other''.

\begin{prop}[{\cite[Proposition~1]{CaFaJi2013}}]
Let $\Theta$ be the angle between two independent uniformly distributed random points $\PT{X}$ and $\PT{Y}$ on $\Sph{d}$ (i.e., $\cos \Theta = \PT{X} \cdot \PT{Y}$). Then
\begin{equation*}
\prob\Big( \left| \Theta - \frac{\pi}{2} \right| \geq \eps \Big) \leq C^\prime \, \sqrt{d} \left( \cos \eps \right)^{d-1}
\end{equation*}
for all $d \geq 1$ and $\eps \in ( 0, \pi / 2 )$, where $C^\prime$ is a universal constant.
\end{prop}

\begin{rmk}
As the dimension $d$ grows (and $\eps$ remains fixed), the probability decays exponentially. Letting $\eps$ decay like $\sqrt{c \, ( \log d ) / d}$, where $c > 1$, one has that (see \cite{CaFaJi2013} for details)
\begin{equation} \label{eq:folklore.concrete}
\prob\Big( \left| \Theta - \frac{\pi}{2} \right| \geq \sqrt{c \, \frac{\log d}{d}} \Big) \leq C^{\prime\prime} \, d^{-(c-1)/2}
\end{equation}
for sufficiently large $d$, where $C^{\prime\prime}$ depends only on $c$; i.e., \eqref{eq:folklore.concrete} can be interpreted as follows: the angle between two independent uniformly distributed points on a high-dimensional sphere is within $\sqrt{( c \, \log d ) / d}$ of $\pi / 2$ with high probability.
\end{rmk}

The minimum geodesic distance among points of a (random) configuration on the sphere is given by the smallest angle. Let the extreme angle $\Theta_{\min}$ be defined by
\begin{equation*}
\sepDist( \{ \PT{X}_1, \dots, \PT{X}_N \} ) \DEF \Theta_{\min} \DEF \min\{ \Theta_{j,k} : 1 \leq j < k \leq N \}.
\end{equation*}
The \MARKED{red}{'extreme law'} of all pairwise angles $\Theta_{j,k}$ among a large number of random points on~$\Sph{d}$ is as follows.


\MARKED{red}{
\begin{thm}[{cf. \cite[Theorem~2]{CaFaJi2013}}] \label{thm:Extreme.law}
\MARKED{red}{Let $X_N$ be a set of $N$ i.i.d. uniformly chosen points on $\Sph{d}$. Denote
\begin{equation} \label{eq:F.N.extreme.law}
F_N( t ) \DEF \prob\big( N^{2/d} \, \sepDist( X_N ) \leq t \big).
\end{equation}
Then for every $t \in \mathbb{R}$, one has $F_N( t ) \to F( t )$, where 
\begin{equation} \label{eq:limiting.CDF}
F( t ) \DEF
\begin{cases}
1 - e^{- \kappa_d \, t^d / 2} & \text{if $t \geq 0$,} \\
0 & \text{if $t < 0$,}
\end{cases}
\end{equation}
and the constant $\kappa_d$ is given in \eqref{eq:Sdconst}.
}
\end{thm}
}

\MARKED{red}{As we shall show, the extreme law for pairwise angles plays an important role in deriving the expected value of the (geodesic) separation distance of random points on $\Sph{d}$.}

\begin{cor} \label{cor:limit.random.separation}
\MARKED{red}{Let $X_N$ be a set of $N$ i.i.d. uniformly chosen points on $\Sph{d}$. Then}
\begin{equation} \label{eq:limit.C.d}
\mathbb{E}[ N^{2/d} \, \MARKED{red}{\sepDist( X_N )} ] \to C_d \DEF \left( \frac{\kappa_d}{2} \right)^{-1/d} \gammafcn( 1 + 1 / d ) \qquad \text{as $N \to \infty$.}
\end{equation}
\MARKED{red}{Furthermore,
\begin{equation*}
\var( N^{2/d} \, \MARKED{red}{\sepDist( X_N )} ) \to \left( \frac{\kappa_d}{2} \right)^{-2/d} \left( \gammafcn( 1 + 2/d ) - [ \gammafcn( 1 + 1/d ) ]^2 \right) \qquad \text{as $N \to \infty$.}
%
\end{equation*}}
\end{cor}
\MARKED{red}{We remark that the same limit relations hold if $\MARKED{red}{\sepDist( X_N )}$ is replaced by the pairwise minimal Euclidean distance of $X_N$. The first few constants $C_d$ are given in Table~\ref{tbl:const.C.d}.}

\begin{table}[ht]
\caption{\label{tbl:const.C.d} The constant $C_d$ \MARKED{red}{in \eqref{eq:limit.C.d}.}} 
\begin{tabular}{ c | c| c| c| c| c| c}
$d$   & $1$     & $2$            & $3$                                  & $4$                                    & $5$ & $6$ \\
\hline
$C_d$ & $2 \pi$ & $\sqrt{2 \pi}$ & $( 3 \pi )^{1/3} \gammafcn( 4 / 3 )$ & $2 ( 2 / 3 )^{1/4} \gammafcn( 5 / 4 )$ & $( 15 \pi / 4 )^{1/5} \gammafcn( 6 / 5 )$ & $( 2 / 5^{1/6} ) \gammafcn( 7 / 6 )$
\end{tabular}
\end{table}

We denote by $A_d(\eps)$ the $\sigma_d$-measure of a spherical cap on $\Sph{d}$ with geodesic radius $\eps$.

\begin{prop} \label{prop:separation.probability}
Assume that $\binom{N}{2} \, A_d(\eps) \leq 1$. Then the geodesic separation distance of $N$ i.i.d. uniformly chosen points on $\Sph{d}$ satisfies
\begin{equation} \label{eq:separation.probability.A}
\begin{split}
\MARKED{red}{\prod_{k=1}^{N-1} \left( 1 - k \, A_d(\eps/2) \right) \geq} \prob( \MARKED{red}{\sepDist( X_N )} \geq \eps ) 
&\geq \MARKED{red}{\prod_{k=1}^{N-1} \left( 1 - k \, A_d(\eps) \right)} \geq 1 - \kappa_d \, \binom{N}{2} \, \eps^d.
\end{split}
\end{equation}
In particular,
\begin{equation} \label{eq:separation.probability.B}
\prob\big( N^{2/d} \, \MARKED{red}{\sepDist( X_N )} \geq \delta \big) \geq 1 - \frac{\kappa_d}{2} \, \delta^d
\end{equation}
and therefore for $\delta = C_d$ and $C_d$ the constant in Corollary~\ref{cor:limit.random.separation},
\begin{equation} \label{eq:separation.probability.C}
\prob\big( N^{2/d} \, \MARKED{red}{\sepDist( X_N )} \geq C_d \big) \geq 1 - \left[ \gammafcn( 1 + \frac{1}{d} ) \right]^{d}.
\end{equation}
The right-hand side is strictly monotonically increasing for $d \geq 1$ towards $1 - e^{\EulerGamma}$. Moreover,
\begin{equation*}
\left[ \gammafcn( 1 + \frac{1}{d} ) \right]^{d} = e^{-\EulerGamma} \left\{ 1 + \frac{\pi^2}{12} \, \frac{1}{d} + \mathcal{O}\big( \frac{1}{d^2} \big) \right\} \to 0.561459483566885\dots \qquad \text{as $d \to \infty$.}
\end{equation*}
(Here, $\EulerGamma$ is the Euler-Mascheroni constant.)
\end{prop}

By the Markov inequality, we have 
\begin{equation*}
\prob\big( N^{2/d} \, \MARKED{red}{\sepDist( X_N )} \geq \delta \big) \leq \frac{\mathbb{E}[ N^{2/d} \, \MARKED{red}{\sepDist( X_N )} ]}{\delta} 
\end{equation*}
and relation \eqref{eq:separation.probability.B} gives that
\begin{equation} \label{eq:expected.normalized.Theta.min.lower.bound}
\mathbb{E}[ N^{2/d} \, \MARKED{red}{\sepDist( X_N )} ] \geq \delta \left( 1 - \frac{\kappa_d}{2} \, \delta^d \right) \qquad \text{subject to $\displaystyle 0 < \delta \leq \left( 2 / \kappa_d \right)^{-1/d}$.}
\end{equation}
Optimization of above lower bound of the expected value yields the following \ estimate that should compared with the one in Corollary~\ref{cor:limit.random.separation}.

\begin{prop} \label{prop:limit.random.separation}
\MARKED{red}{Let $X_N$ be a set of $N$ i.i.d. uniformly chosen points on $\Sph{d}$. Then}
\begin{equation}
\mathbb{E}[ N^{2/d} \, \MARKED{red}{\sepDist( X_N )} ] \geq C_d \, \frac{\left( d + 1 \right)^{-1/d}}{\gammafcn( 2 + 1 / d )} \qquad \text{for every $N \geq 2$,}
\end{equation}
where $C_d$ is the constant in Corollary~\ref{cor:limit.random.separation} and
\begin{equation*}
\frac{\left( d + 1 \right)^{-1/d}}{\gammafcn( 2 + 1 / d )} = 1 - \frac{\log d}{d} - \frac{1-\EulerGamma}{d} + \mathcal{O}\big( \frac{1}{d^2} \big) \qquad \text{as $d \to \infty$.}
\end{equation*}
\end{prop}

\section{Proofs}
\label{sec:proofs}

\subsection{Proofs of Section~\ref{sec:covering.random.points}}

\begin{proof}[Proof of Proposition~\ref{prop:holes.radii}]
The volume of an upright hyper-pyramid in $\mathbb{R}^{d+1}$ of height $a$ with a base polytope of $d$-dimensional surface measure $A$ has volume $A a / ( d + 1 )$. Hence, the volume $V_N$ of a polytope with $N$ vertices on $\Sph{d}$ that contains the origin in its interior with $f_d$ facets $F_1, \dots, F_{f_d}$, where $F_k$ has surface area $A_{k,N}$ and distance $a_k$ to the origin, is given by the formula
\begin{equation} \label{eq:polytope.volume.formula}
V_N = \frac{1}{d+1} \sum_{k=1}^{f_d} A_{k,N} \, a_k.
\end{equation}
A simple geometric argument yields that the geodesic radius $\alpha_k$ and the Euclidean radius $\rho_k$ of the ``point-free'' spherical cap associated with $F_k$ are related with the center distance $a_k$ by means of $\rho_k = 2 \sin( \alpha_k / 2 ) = \sqrt{2\left( 1 - a_k \right)}$.

Now, let $V_N$ and $A_N$ be the volume and the surface area of the convex hull $K_N$ of the random points $\PT{Y}_1, \dots, \PT{Y}_N$. Further, let $V$ and $A$ denote the volume of the unit ball in $\mathbb{R}^{d+1}$ and the surface area of its boundary $\Sph{d}$, respectively.
We will use the following facts: $( d + 1 ) V = A$, $(d + 1) V_N \leq A_N$, and that $V_N / V \to 1$ and $A_N / A \to 1$ as $N \to \infty$, which are consequences of \eqref{eq:polytope.volume.formula} and the fact that the random polytopes $K_N$ approximate the unit ball as $N \to \infty$. We rewrite \eqref{eq:polytope.volume.formula},
\begin{align*}
1 - \frac{V_N}{V}
&= 1 - \frac{A_N}{\left( d + 1 \right) V} \sum_{k=1}^{f_d} \frac{A_{k,N}}{A_N} \, a_k \\
&= 1 - \sum_{k=1}^{f_d} \frac{A_{k,N}}{A_N} \, a_k - \left( \frac{A_N}{A} - 1 \right) \sum_{k=1}^{f_d} \frac{A_{k,N}}{A_N} \, a_k \\
&= \sum_{k=1}^{f_d} \frac{A_{k,N}}{A_N} \left( 1 - a_k \right) - \left( \frac{A_N}{A} - 1 \right) \frac{\left( d + 1 \right) V_N}{A_N} \\
&= \frac{1}{2} \, \sum_{k=1}^{f_d} \frac{A_{k,N}}{A_N} \, \rho_k^2 - \left( \frac{A_N}{A} - 1 \right) - \left( \frac{A_N}{A} - 1 \right) \left( \frac{\left( d + 1 \right) V_N}{A_N} - 1 \right).
\end{align*}
Hence, we arrive at the expected value for the measure $\mu$ given by $\dd \mu(\PT{x}) = \eta(\PT{x}) \dd \sigma_d(\PT{x})$
\begin{equation}
\begin{split}
\mathbb{E}_{\mu}\left[ \sum_{k=1}^{f_d} \frac{A_{k,N}}{A_N} \, \rho_k^2 \right]
&= 2 \, \mathbb{E}_{\mu}\left[ 1 - \frac{V_N}{V} \right] -
   2 \mathbb{E}_{\mu}\left[ 1 - \frac{A_N}{A} \right] \\
&\phantom{=}+ 2 \, \mathbb{E}_{\mu}\left[ \left( 1 - \frac{A_N}{A} \right) \left( 1 - \frac{(d + 1) V_N}{A_N} \right) \right].
\end{split}
\end{equation}
Observe that the last term is positive. Furthermore, $( d + 1 ) V_N / A_N = ( V_N / V ) \big/ ( A_N / A )$.

From \cite[Theorem~1.1]{BoKaFeHu2013} (also cf. \cite{Mu1990,Rei2002,SchWe2003}), we obtain that 
\begin{equation*}
2 \, \mathbb{E}_{\mu}\left[ 1 - \frac{V_N}{V} \right] \sim \frac{1}{N^{2/d}} \, \frac{\gammafcn( d + 2 + 2/d )}{(d-1)!\left(d+2\right)} \left( \kappa_d \right)^{-2/d} \int_{\Sph{d}} \left[ \eta( \PT{x} ) \right]^{-2/d} \dd \sigma_d( \PT{x} ),
\end{equation*}
and
\begin{equation*}
2 \, \mathbb{E}_{\mu}\left[ 1 - \frac{A_N}{A} \right] \sim \frac{1}{N^{2/d}} \, \frac{d \, \gammafcn(d + 1 + 2/d)}{(d-1)!\left(d+2\right)} \left( \kappa_d \right)^{-2/d} \int_{\Sph{d}} \left[ \eta( \PT{x} ) \right]^{-2/d} \dd \sigma_d( \PT{x} ).
\end{equation*}
The result follows.
\end{proof}

For the proof of Theorem~\ref{thm:sums.of.powers.of.holes.radii} we need the following asymptotic result.

\begin{lem} \label{lem:asymptotics}
Let $\nu \geq 1$, $0 < c \leq 1$, $\mu > 0$, $\beta > 0$, and $\tau > 0$ be such that $c \, \tau^\mu < 1$ and $\nu/\mu \geq 1$. Suppose that $g$ is a continuous function on $[0,\tau]$ which is differentiable on $(0,\tau)$, and satisfies $| g(x) | \leq C \, x^\kappa$ and $| g^\prime( x ) | \leq C \, x^{\kappa-1}$ on $[0,\tau]$ for some constants $C > 0$ and $\kappa > 0$. Then for $N \geq \beta + 2$,
\begin{align}
&\int_0^\tau x^{\nu - 1} \left[ 1 - c \, x^\mu \left\{ 1 + g( x ) \right\} \right]^{N-\beta-1} \dd x \notag \\
&\phantom{equals}= \frac{c^{-\nu/\mu}}{\mu} \frac{\gammafcn( \nu / \mu ) \gammafcn( N - \beta )}{\gammafcn( N - \beta + \nu / \mu)}
\IncompleteBetaRegularized_{c \, \tau^\mu}( \nu / \mu,  N - \beta) +
  \mathcal{R}_{c \, \tau^\mu}( \nu / \mu, N - \beta ) \label{eq:asymptotics.A} \\
&\phantom{equals}\sim \frac{c^{-\nu/\mu}}{\mu} \gammafcn( \nu / \mu ) \, N^{-\nu/ \mu} \left\{ 1 + \mathcal{O}( N^{-\kappa/\mu} ) \right\} \qquad \text{as $N \to \infty$.}
\end{align}
The remainder term $\mathcal{R}_{c \, \tau^\mu}( \nu / \mu, N - \beta )$, given by
\begin{equation*}
\mathcal{R}_{c \, \tau^\mu}( \nu / \mu, N - \beta ) = \frac{c^{-\nu/\mu}}{\mu} \int_0^{c \, \tau^\mu} h(u) \, u^{\nu/\mu - 1} \left( 1 - u \right)^{N - \beta - 1} \dd u,
\end{equation*}
where
\begin{equation*}
h( u ) \DEF \left[ 1 - \frac{u}{1-u} \, g\Big( \left( \frac{u}{c} \right)^{1/\mu} \Big) \right]^{N-\beta-1} - 1,
\end{equation*}
satisfies the estimate
\begin{equation} \label{eq:asymptotics.R}
\left| \mathcal{R}_{c \, \tau^\mu}( \nu / \mu, N - \beta ) \right| \leq \frac{C}{1 - c \, \tau^\mu} \, \frac{1+\mu}{\mu^2} \frac{c^{-(\nu+\kappa)/\mu}}{\left( 1 - C\,
\tau^\kappa \right)^{1+(\nu+\kappa)/\mu}} \, \frac{\gammafcn( 1 + ( \nu + \kappa ) / \mu )}{\left( N - \beta \right)^{( \nu + \kappa ) / \mu }},
\end{equation}
provided $C \, \tau^\kappa < 1$.

Furthermore, if $g \equiv 0$, then $\mathcal{R}_{c \, \tau^\mu}( \nu / \mu, N - \beta )$ vanishes and one has the more precise asymptotic expansion
\begin{equation}
\begin{split} \label{eq:more.precise.asymptotics.B}
&\int_0^\tau x^{\nu - 1} \left( 1 - c \, x^\mu \right)^{N-\beta-1} \dd x \\
&\phantom{equ}\sim \frac{c^{-\nu/\mu}}{\mu} \gammafcn( \nu / \mu ) \, N^{-\nu/ \mu} \left\{ 1 + \frac{1}{2} \, \frac{\nu}{\mu} \left( 2 \beta + 1 - \frac{\nu}{\mu} \right) \frac{1}{N} + \mathcal{O}\left( \frac{1}{N^2} \right) \right\} \quad \text{as $N \to \infty$.}
\end{split}
\end{equation}
\end{lem}

\begin{proof}
Let $G_\tau( \nu, N - \beta )$ denote the integral on the LHS of \eqref{eq:asymptotics.A}.
Then the change of variable $u = c x^\mu$ gives an incomplete beta-function-like integral
\begin{equation*}
G_\tau( \nu, N - \beta ) =
  \frac{c^{-\nu/\mu}}{\mu} \int_0^{c \, \tau^\mu} u^{\nu/\mu-1}
  \left[ 1 - u \left\{ 1 + g\Big( \left( \frac{u}{c} \right)^{1/\mu} \Big) \right\} \right]^{N-\beta-1} \dd u,
\end{equation*}
which can be rewritten as follows
\begin{equation*}
  G_\tau( \nu, N - \beta ) = \frac{c^{-\nu/\mu}}{\mu}
  \frac{\gammafcn( \nu / \mu ) \gammafcn( N - \beta )}{\gammafcn( N - \beta + \nu / \mu)} \IncompleteBetaRegularized_{c \, \tau^\mu}(\nu/\mu, N-\beta) +
  \mathcal{R}_{c \, \tau^\mu}( \nu / \mu, N - \beta ).
\end{equation*}
The ratio of Gamma functions has the asymptotic expansion (\cite[Eqs.~5.11.13 and 5.11.15]{NIST:DLMF})
\begin{equation} \label{eq:gamma.ratio.asymptotics}
\frac{\gammafcn( N - \beta )}{\gammafcn( N - \beta + \nu / \mu)} = N^{-\nu/\mu} \left\{ 1 + \frac{1}{2} \, \frac{\nu}{\mu} \left( 2 \beta + 1 - \frac{\nu}{\mu} \right) \frac{1}{N} + \mathcal{O}\left( \frac{1}{N^2} \right) \right\} \quad \text{as $N \to \infty$,}
\end{equation}
whereas the regularized incomplete beta function admits the asymptotics (\cite[Eq.~8.18.1]{NIST:DLMF})
\begin{equation*}
\begin{split}
&1-\IncompleteBetaRegularized_{c \, \tau^\mu}( \nu / \mu , N - \beta)
= \left( 1 - c \, \tau^\mu \right)^{N-\beta} \left( c \, \tau^\mu \right)^{\nu / \mu - 1} \\
&\phantom{equals\pm}\times \left\{ \sum_{k = 0}^{n - 1} \frac{\gammafcn( N - \beta + \nu / \mu )}{\gammafcn( N - \beta + k + 1 ) \gammafcn( \nu / \mu - k )} \left( \frac{1 - c \, \tau^\mu}{c \, \tau^\mu} \right)^k + \mathcal{O}\Big( \frac{\gammafcn( N - \beta + \nu / \mu )}{\gammafcn( N - \beta + n + 1 )} \Big) \right\}
\end{split}
\end{equation*}
so that the departure from 1 has exponential decay as $N \to \infty$. (Note that the expression in curly braces is exact [with $\mathcal{O}$-term omitted] if $\nu / \mu$ is a positive integer and $n \geq \nu / \mu$.)

It remains to investigate the term $\mathcal{R}_{c \, \tau^\mu}( \nu /\mu, N - \beta )$. By the mean value theorem
$h( u ) = u \, h^\prime( \xi )$ for some $0 < \xi = \xi(u) < u$. From
\begin{equation*}
\begin{split}
h^\prime( u )
&= \left( N - \beta - 1 \right) \left[ 1 - \frac{u}{1-u} \, g\Big( \left( \frac{u}{c} \right)^{1/\mu} \Big) \right]^{N-\beta-2} \\
&\phantom{=\pm}\times \left\{ - \frac{g( (u/c)^{1/\mu} )}{\left( 1 - u \right)^2} - \frac{1}{\mu} \, \frac{(u/c)^{1/\mu} \, g^\prime( (u/c)^{1/\mu} )}{1-u} \right\}
\end{split}
\end{equation*}
and the assumptions regarding $g$ it follows that 
\begin{equation*}
\left| h^\prime( \xi ) \right| \leq \frac{C}{1 - c \, \tau^\mu} \frac{1+\mu}{\mu} \left( N - \beta - 1 \right) \left[ 1 + \frac{u}{1-u} \, C \left( \frac{u}{c} \right)^{\kappa/\mu} \right]^{N-\beta-2} \frac{\left( u / c \right)^{\kappa/\mu}}{1 - u} 
\end{equation*}
for $u \in [0, c \, \tau^\mu]$. (The square-bracketed expression can be omitted if $g$ is positive. This observation simplifies the next estimates and removes the requirement that $C \, \tau^{\kappa} < 1$.) Hence
\begin{equation*}
\left| \mathcal{R}_{c \, \tau^\mu}( \nu / \mu, N - \beta ) \right| \leq \frac{C}{1 - c \, \tau^\mu} \, \frac{c^{-(\nu+\kappa)/\mu}}{\mu} \frac{1+\mu}{\mu} \left( N - \beta - 1 \right) H_{c \, \tau^\mu}( \nu / \mu, N - \beta ),
\end{equation*}
where the change of variable $v = ( 1 - C \, \tau^\kappa ) u$ gives the desired incomplete beta function:
\begin{align*}
H_{c \, \tau^\mu}( \nu / \mu, N - \beta )
&\DEF \int_0^{c \, \tau^\mu} u^{(\nu+\kappa)/\mu} \left[ 1 - u \left\{ 1 - C \, \tau^\kappa \right\} \right]^{N - \beta - 2} \dd u \\
&= \frac{1}{\left( 1 - C\,\tau^\kappa \right)^{1+(\nu+\kappa)/\mu}} \int_{0}^{\left( 1 - C \, \tau^\kappa \right) \, c \, \tau^\mu} v^{1 + (\nu+\kappa)/\mu - 1} \left( 1 - v \right)^{N - \beta - 1 - 1} \dd v \\
&\leq \frac{1}{\left( 1 - C\,\tau^\kappa \right)^{1+(\nu+\kappa)/\mu}} \, \frac{\gammafcn( 1 + ( \nu + \kappa ) / \mu ) \gammafcn( N - \beta - 1 )}{\gammafcn( N - \beta + ( \nu + \kappa ) / \mu )}, 
\end{align*}
and hence
\begin{equation*}
\begin{split}
\left| \mathcal{R}_{c \, \tau^\mu}( \nu / \mu, N - \beta ) \right|
&\leq \frac{C}{1 - c \, \tau^\mu} \, \frac{c^{-(\nu+\kappa)/\mu}}{\mu} \frac{1+\mu}{\mu} \frac{N - \beta - 1}{\left( 1 - C\,\tau^\kappa \right)^{1+(\nu+\kappa)/\mu}} \\
&\phantom{equals}\times \frac{\gammafcn( 1 + ( \nu + \kappa ) / \mu ) \gammafcn( N - \beta - 1)}{\gammafcn( N - \beta + ( \nu + \kappa ) / \mu )}.
\end{split}
\end{equation*}
Using the estimate \cite[Eq.~5.6.8]{NIST:DLMF} for ratios of
gamma functions, valid for $N > \beta+1$ and $(\nu + \kappa)/\mu \geq 1$, we arrive at the result.
Furthermore, if $g \equiv 0$, then ${\mathcal{R}_{c \, \tau^\mu}( \nu / \mu, N - \beta )}$ vanishes and \eqref{eq:gamma.ratio.asymptotics} implies the more precise asymptotic expansion given in \eqref{eq:more.precise.asymptotics.B}.
\end{proof}

\begin{proof}[Proof of Theorem~\ref{thm:sums.of.powers.of.holes.radii}]
Let $K_N$ be the convex hull of $N$ points on $\Sph{d}$ that contains the origin. The $k$th facet with distance $a_k$ to the origin determines a hole with Euclidean hole radius $\rho_k$ with $\rho_k^2 = 2 - 2 a_k$. Hence, the sum of the $p$th powers (${p > 0}$) of the $f_d$ hole radii $\rho_1, \dots, \rho_{f_d}$ is given by
\begin{equation*}
\sum_{k=1}^{f_d} \left( \rho_k \right)^p = \sum_{k=1}^{f_d} \left( 2 - 2 \, a_k \right)^{p/2}.
\end{equation*}
(If the origin is not contained in the convex hull and, say, the $k^*$th facet is closest to the origin [there is only one such facet], then the term $( 2 - 2 \, a_{k^*} )^{p/2}$ needs to be replaced with $( 2 + 2 \, a_{k^*} )^{p/2}$. In such a case $\rho_{k^*}$ is also the covering radius of the $N$ points.)

Now let $K_N$ be the convex hull of the $N$ random points $\PT{X}_1, \dots, \PT{X}_N$. First observe that
\begin{equation*}
\begin{split}
\mathbb{E}\left[ \sum_{k = 1}^{f_d} \left( \rho_k \right)^p \right]
&= \int_{\Sph{d}} \cdots \int_{\Sph{d}} \sum_{k=1}^{f_d} \left( 2 - 2 \, a_k \right)^{p/2} \dd \sigma_d( \PT{x}_1) \cdots \sigma_d( \PT{x}_N) \\
&\phantom{=}+ \mathop{\int_{\Sph{d}} \cdots \int_{\Sph{d}}}_{\PT{0} \notin K_N} \left[ \left( 2 + 2 \, a_* \right)^{p/2} - \left( 2 - 2 \, a_* \right)^{p/2} \right] \dd \sigma_d( \PT{x}_1) \cdots \sigma_d( \PT{x}_N),
\end{split}
\end{equation*}
where $a_* = a_*( \PT{X}_1, \dots, \PT{X}_{N} )$ is the minimum of the distances from the origin to the facets of $K_N$.
For the computation of the asymptotic form of the expected value we adapt the approach in \cite{BuMuTi1985} and \cite{Mu1990}.
As all the vertices of $K_N$ are chosen independently and uniformly on $\Sph{d}$, the probability that the convex hull of $\PT{X}_1, \dots, \PT{X}_{d+1}$ forms a $d$-dimensional facet ${F = F( \PT{X}_1, \dots, \PT{X}_{d+1} )}$ of $K_N$ is
\begin{equation*}
\lambda^{N-d-1} + \left( 1 - \lambda \right)^{N-d-1} \qquad \text{for $N \geq d + 2$,}
\end{equation*}
where $\lambda = \lambda( \PT{X}_1, \dots, \PT{X}_{d+1} )$ is the $\sigma_d$-surface area of the smaller of the two spherical caps due to the intersection of $\Sph{d}$ with the supporting hyperplane of $\PT{X}_1, \dots, \PT{X}_{d+1}$. This is because points $\PT{X}_1,\ldots,\PT{X}_{d+1}$ form a facet of $K_N$ if and only if all subsequent points $\PT{X}_{d+2},\ldots,\PT{X}_N$ fall on the same side of the supporting hyperplane. There are $\binom{N}{d+1}$ possibilities of selecting $d+1$ points out of $N$.
The probability that the origin is not an interior point of $K_N$ tends to zero exponentially fast as $N \to \infty$ (cf. \cite{We1962}), so that
\begin{equation*}
\begin{split}
\mathbb{E}\left[ \sum_{k = 1}^{f_d} \left( \rho_k \right)^p \right]
&\sim \binom{N}{d+1} \int_{\Sph{d}} \cdots \int_{\Sph{d}} \left( \lambda^{N-d-1} + \left( 1 - \lambda \right)^{N-d-1} \right) \\
&\phantom{equalsequals\pm}\times \left( 2 - 2 a \right)^{p/2} \dd \sigma_d( \PT{x}_1 ) \cdots \dd \sigma_d( \PT{x}_{d+1} ) \qquad \text{as $N \to \infty$}
\end{split}
\end{equation*}
and the error thus introduced is negligible compared to the leading terms in the asymptotics. Here $a = a( \PT{X}_1, \dots, \PT{X}_{d+1} )$ denotes the distance of the facet $F$ from the origin. The integral can be rewritten using a stochastically equivalent sequential method to choose $d+1$ points independently and uniformly on $\Sph{d}$ (cf. \cite{BuMuTi1985} and, in particular, \cite[Theorem~4]{Mi1971}). This method utilizes in the first step a sphere intersecting a random hyperplane with unit normal vector $\PT{u}$ uniformly chosen from $\Sph{d}$ where the distance $a$ of the plane to the origin is distributed according to the probability density function
\begin{equation*}
\frac{2}{\betafcn(d^2/2,1/2)} \left( 1 - a^2 \right)^{d^2/2 - 1}, \qquad 0 \leq a \leq 1,
\end{equation*}
where $\betafcn(a,b)$ denotes the beta function. In the second independent step choose $d+1$ random points $\overline{\PT{x}}_1, \dots, \overline{\PT{x}}_{d+1}$ from the intersection of the hyperplane and $\Sph{d}$ (which is a $(d-1)$-dimensional sphere of radius $\sqrt{1-a^2}$) so that the density transformation
\begin{equation*}
\begin{split}
\dd \sigma_d( \PT{x}_1 ) \cdots \dd \sigma_d( \PT{x}_{d+1} )
&=  d! \frac{\left( \omega_{d-1} \right)^{d+1}}{\left( \omega_d \right)^d} \left( 1 - a^2 \right)^{(d+1)(d-2)/2} T \\
&\phantom{=\pm}\times \frac{\dd \omega^\prime( \overline{\PT{x}}_1 )}{\left( 1 - a^2 \right)^{(d-1)/2} \omega_{d-1}} \cdots \frac{\dd \omega^\prime( \overline{\PT{x}}_{d+1} )}{\left( 1 - a^2 \right)^{(d-1)/2} \omega_{d-1}} \, \dd a \, \dd \sigma_d( \PT{u} )
\end{split}
\end{equation*}
applies, where $T$ is the $d$-dimensional volume of the convex hull of the points $\overline{\PT{x}}_1, \dots, \overline{\PT{x}}_{d+1}$ and $\omega^\prime$ denotes the surface area measure of the intersection of $\Sph{d}$ and the hyperplane. Furthermore, $\lambda$ is the $\sigma_d$-measure of the spherical cap $\{ \PT{x} \in \Sph{d} : \PT{x} \cdot \PT{u} \geq a \}$ and depends on $a$ only; i.e., application of the Funk-Hecke formula (cf. \cite{Mu1966}) yields for $0 \leq a \leq 1$
\begin{align}
\lambda &= \frac{\omega_{d-1}}{\omega_d} \int_a^1 \left( 1 - t^2 \right)^{d/2 - 1} \dd t = \frac{1}{2} - \frac{\omega_{d-1}}{\omega_d} \int_0^a \left( 1 - t^2 \right)^{d/2 - 1} \dd t, \label{eq:lambda.int.represent}\\
1 - \lambda &= \frac{\omega_{d-1}}{\omega_d} \int_{-1}^a \left( 1 - t^2 \right)^{d/2 - 1} \dd t = \frac{1}{2} + \frac{\omega_{d-1}}{\omega_d} \int_0^a \left( 1 - t^2 \right)^{d/2 - 1} \dd t.\notag
\end{align} 
Thus, we arrive at
\begin{equation*}
\begin{split}
\mathbb{E}\left[ \sum_{k = 1}^{f_d} \left( \rho_k \right)^p \right]
&\sim \binom{N}{d+1} \, d! \frac{\left( \omega_{d-1} \right)^{d+1}}{\left( \omega_d \right)^d} \, \int_{\Sph{d}} \int_0^1 \left( \lambda^{N-d-1} + \left( 1 - \lambda \right)^{N-d-1} \right) \\
&\phantom{equalsequals\pm}\times \left( 2 - 2 a \right)^{p/2} \left( 1 - a^2 \right)^{(d+1)(d-2)/2} m_1( a ) \dd a \dd \sigma_d( \PT{u} ), 
\end{split}
\end{equation*}
where $m_1(a)$ is the first moment of the $d$-dimensional volume of the convex hull of ${d+1}$ points chosen from the $(d-1)$-dimensional sphere with radius $r = \sqrt{1 - a^2}$. From \cite[Theorem~2]{Mi1971}, we obtain
\begin{equation*}
m_1( a ) = \frac{2}{d!} \, \frac{1}{\betafcn( d^2/2, 1/2 )} \, \frac{\left( \omega_{d} \right)^d}{\left( \omega_{d-1} \right)^{d+1}} \left( 1 - a^2 \right)^{d/2}.
\end{equation*}
Thus ($\int_{\Sph{d}} \dd \sigma_d = 1$)
\begin{equation*}
\mathbb{E}\left[ \sum_{k = 1}^{f_d} \left( \rho_k \right)^p \right] \sim \frac{\displaystyle 2 \binom{N}{d+1}}{\betafcn( d^2/2, 1/2 )} \, \int_0^1 \left( \lambda^{N-d-1} + \left( 1 - \lambda \right)^{N-d-1} \right) \left( 2 - 2 \, a \right)^{p/2} \left( 1 - a^2 \right)^{d^2/2 - 1} \dd a. 
\end{equation*}
The change of variable $2 ( 1 - a ) = \rho^2$ introduces the Euclidean radius of the spherical cap with $\sigma_d$-surface area $\lambda$ into the integral; i.e.,
\begin{equation*}
\mathbb{E}\left[ \sum_{k = 1}^{f_d} \left( \rho_k \right)^p \right] \sim \frac{\displaystyle 2 \binom{N}{d+1}}{\betafcn( d^2/2, 1/2 )}  \, \int_0^{\sqrt{2}} \rho^p \left( \lambda^{N-d-1} + \left( 1 - \lambda \right)^{N-d-1} \right) \rho^{d^2-1} \left( 1 - \frac{\rho^2}{4} \right)^{d^2/2 - 1} \dd \rho 
\end{equation*}
and $\lambda$ can be represented by \eqref{eq:lambda.int.represent} as
\begin{equation} \label{eq:lambda}
\lambda = \IncompleteBetaRegularized_{\rho^2/4}( d/2, d/2 ) = \frac{1}{d} \frac{\omega_{d-1}}{\omega_d} \, \rho^d \, \Hypergeom{2}{1}{1-d/2,d/2}{1+d/2}{ \frac{\rho^2}{4}}.
\end{equation}
The last hypergeometric function reduces to a polynomial of degree $d/2-1$ for even $d$.
Since $\lambda \leq 1/2$ for all $\rho \in [0,\sqrt{2}]$, the contribution due to $\lambda^{N-d-1}$ decays exponentially fast; i.e.,
\begin{equation*}
\begin{split}
&\frac{2}{\betafcn( d^2/2, 1/2 )} \binom{N}{d+1} \, \int_0^{\sqrt{2}} \rho^p \, \lambda^{N-d-1} \, \rho^{d^2-1} \left( 1 - \frac{\rho^2}{4} \right)^{d^2/2 - 1} \dd \rho \\
&\phantom{equals}\leq \left( \frac{1}{2} \right)^{N-d-1} \frac{2^{p+d^2}}{\betafcn( d^2/2, 1/2 )} \binom{N}{d+1} \, \int_0^{1/2} u^{p/2} \left[ u \left( 1 - u \right) \right]^{d^2/2-1} \dd u.
\end{split}
\end{equation*}
Consequently, up to an exponentially fast decaying contribution,
\begin{equation} \label{eq:auxiliary.A}
\mathbb{E}\left[ \sum_{k = 1}^{f_d} \left( \rho_k \right)^p \right] \sim \frac{\displaystyle 2 \binom{N}{d+1}}{\betafcn( d^2/2, 1/2 )}  \, \int_0^{\sqrt{2}} \rho^p \left( 1 - \lambda \right)^{N-d-1} \, \rho^{d^2-1} \left( 1 - \frac{\rho^2}{4} \right)^{d^2/2 - 1} \dd \rho. 
\end{equation}

{\bf Case $d = 2$.} In this case $\lambda =\tfrac{1}{2}(1-a) = \tfrac{1}{2}(1-\rho^2/4)$, and the expected value formula simplify further to
\begin{align*}
\mathbb{E}\left[ \sum_{k = 1}^{f_2} \left( \rho_k \right)^p \right]
&\sim \frac{3}{2} \binom{N}{3} \int_0^{\sqrt{2}} \rho^{p + 3} \left( 1 - \frac{\rho^2}{4} \right)^{N-2} \dd \rho.
\end{align*}
Using Lemma~\ref{lem:asymptotics} with $\nu = p + 4$, $c = 1/4$, $\mu = 2$, $\beta = 1$, and $\tau = \sqrt{2}$ (so that $c \, \tau^\mu = 1/2$), we arrive after some simplifications at
\begin{equation*}
\mathbb{E}\left[ \sum_{k = 1}^{f_2} \left( \rho_k \right)^p \right] \sim \left( 2 N - 4 \right) 2^{p} \frac{\gammafcn( 2 + \frac{p}{2} ) \gammafcn( N + 1 )}{\gammafcn( N + 1 + p/2 )} \left\{ 1 - \IncompleteBetaRegularized_{1/2}( N - 1, 2 + p/2 ) \right\}.
\end{equation*}
The regularized incomplete beta function decays exponentially fast as $N \to \infty$. The classical asymptotic formula for ratios of gamma functions (\cite[Eq.s~5.11.13 and 5.11.17]{NIST:DLMF}) yields
\begin{equation*}
\mathbb{E}\left[ \sum_{k = 1}^{f_2} \left( \rho_k \right)^p \right] \sim \left( 2 N - 4 \right) 2^{p} \gammafcn( 2 + \frac{p}{2} ) \, N^{-p/2} \left\{ 1 + \sum_{\ell=1}^{L-1} \frac{\binom{-p/2}{\ell} \, B_\ell^{(1-p/2)}(1)}{N^\ell} + \mathcal{O}\left( \frac{1}{N^L} \right) \right\},
\end{equation*}
where $B_\ell^{(a)}(x)$ are the generalized Bernoulli polynomials.

{\bf General case $d \geq 2$.}
From \eqref{eq:auxiliary.A} we get
\begin{equation*}
\mathbb{E}\left[ \sum_{k = 1}^{f_d} \left( \rho_k \right)^p \right] \sim \frac{2}{\betafcn( d^2/2, 1/2 )} \, \binom{N}{d+1} \, \int_0^{\sqrt{2}} \rho^{p + d^2 - 1} \left( 1 - \lambda \right)^{N-d-1} \left\{ 1 + \mathcal{O}(\rho^2) \right\} \dd \rho,
\end{equation*}
where $\lambda$ is given in \eqref{eq:lambda}.
We want to apply Lemma~\ref{lem:asymptotics}. Observe that
\begin{equation*}
  \kappa_d  \DEF \frac{1}{d} \, \frac{\omega_{d-1}}{\omega_d} =
  \frac{1}{d} \, \frac{\gammafcn( (d+1)/2 )}{\sqrt{\pi} \, \gammafcn( d/2 )} <
  \sqrt{\frac{1}{2\pi} \, \frac{d+1}{d^2}} \leq \frac{1}{\sqrt{\pi}} < 1, \qquad d \geq 1,
\end{equation*}
by \cite[Eq.~5.6.4]{NIST:DLMF}. Furthermore, the continuously differentiable auxiliary function
\begin{equation*}
h( u ) \DEF \Hypergeom{2}{1}{1-d/2,d/2}{1+d/2}{u} - 1 
\end{equation*}
satisfies $h( u ) = u \, h^\prime( \xi )$ for some $\xi = \xi( u ) \in (0, u)$ for every $u \in (0,1)$ by the mean value theorem, whereas
\begin{equation*}
h^\prime( u ) = - \frac{(d/2-1) (d/2)}{d/2 + 1} \, \Hypergeom{2}{1}{2-d/2,1+d/2}{2+d/2}{u}. 
\end{equation*}
Clearly, $| h( u ) | \leq C \, u$ and $| h^\prime( u ) | \leq C$ on $[0,1]$ for
$C \DEF \max_{u \in [0,1]} | h^\prime( u ) |$ if $d \geq 3$ (and $h \equiv 0$ for $d = 2$).
 Letting $g( \rho ) \DEF h( \rho^2 / 4 )$, it follows that $| g( \rho ) | \leq C \, \rho^2 / 4$
and $| g^\prime( \rho ) | \leq C \, \rho / 2$ on $[0,2]$.
Hence, we can apply Lemma~\ref{lem:asymptotics} with $\nu = p + d^2$, $c \DEF \kappa_d$,
$\mu = d$, $\beta = d$, $\kappa = 2$, and $\tau \in ( 0, \sqrt{2} )$
restricted such that $\kappa_d \, \tau^d < 1$ and $C \, \tau^2 < 1$, to obtain
\begin{equation}
\begin{split}
&\mathbb{E}\left[ \sum_{k = 1}^{f_d} \left( \rho_k \right)^p \right]
  \sim \frac{2}{d} \, \frac{\gammafcn( (d^2 + 1)/2 )}{\sqrt{\pi} \, \gammafcn( d^2 / 2 )} \left( \kappa_d \right)^{-d-p/d} \,
    \binom{N}{d+1} \, \frac{\gammafcn( d + p / d ) \gammafcn( N - d )}{\gammafcn( N + p / d )} \\
& \phantom{equals} \times \left\{ 1 - \IncompleteBetaRegularized_{1 - c \, \tau^\mu}( N - d, d + \frac{p}{d}) \right\}
    + \frac{\displaystyle 2 \binom{N}{d+1}}{\betafcn( d^2/2, 1/2 )} \, \mathcal{R}_{1 - c \, \tau^\mu}( d + \frac{p}{d}, N - d ),
\end{split} \label{eq:main.E.sum.relation}
\end{equation}
where we omitted the integral over $[\tau,\sqrt{2}]$ that goes exponentially to zero as $N\to \infty$. Thus
\begin{equation*}
\mathbb{E}\left[ \sum_{k = 1}^{f_d} \left( \rho_k \right)^p \right]
  \sim \frac{2}{d} \, \frac{\gammafcn( (d^2 + 1)/2 )}{\sqrt{\pi} \, \gammafcn( d^2 / 2 )} \left( \kappa_d \right)^{-d-p/d}
    \frac{\gammafcn( d + p/d )}{(d+1)!} \, \frac{\gammafcn( N + 1 )}{\gammafcn( N + p/d )} \left\{ 1 + \mathcal{O}( N^{-2/d} ) \right\}.
\end{equation*}
This shows the first part of the formula. The second form follows when substituting the asymptotic expansion of the ratio $\gammafcn( N + 1 ) / \gammafcn( N + p/d )$. 
\end{proof}


\subsection{Proofs of Section~\ref{sec:separation.random.points}}

For the proof of Corollary~\ref{cor:limit.random.separation} we need the following result.

\begin{lem} \label{lem:aux.expected}
The expected value and the variance of the random variable $Y$ with CDF \eqref{eq:limiting.CDF} are
\begin{equation}
\begin{split} \label{eq:expected.value.Y}
\mathbb{E}[ Y ]
&= \int_0^\infty x \dd F(x) = \left( \frac{1}{2} \kappa_d \right)^{-1/d} \gammafcn( 1 + 1 / d ) \\
&= 1 + \frac{\log \sqrt{d}}{d} + \frac{\log \sqrt{8 \pi} - \EulerGamma}{d} + \mathcal{O}( \frac{1}{d^2} ) \qquad \text{as $d \to \infty$,}
\end{split}
\end{equation}
where $\EulerGamma$ is the Euler-Mascheroni constant and
\begin{equation}
\begin{split} \label{eq:variance.Y}
\var( Y )
&= \int_0^\infty \left( x - \mathbb{E}[ Y ] \right)^2 \dd F(x) = \left( \mathbb{E}[ Y ] \right)^2 \left( \frac{\gammafcn( 1 + 2/d )}{[ \gammafcn( 1 + 1/d ) ]^2} - 1 \right) \\
&= \frac{\pi^2}{6} \, \frac{1}{d^2} + \frac{\pi^2}{6} \, \frac{\log d}{d^3} + \mathcal{O}( \frac{1}{d^3} ) \qquad \text{as $d \to \infty$.}
\end{split}
\end{equation}
\end{lem}

\begin{proof}
Let $R > 0$. Integration by parts gives
\begin{equation*}
\begin{split}
\int_0^R x \dd F( x )
&= R \left( 1 - e^{- ( \kappa_d / 2 ) \, R^d } \right) - \int_0^R F( x ) \dd x = - R \, e^{-( \kappa_d / 2 ) \, R^d} + \int_0^R e^{-( \kappa_d / 2 ) \, x^d} \dd x \\
&\to \int_0^\infty e^{-( \kappa_d / 2 ) \, x^d} \dd x \qquad \text{as $R \to \infty$.}
\end{split}
\end{equation*}
The last integral represents a gamma function (cf. \cite[Eq.~5.9.1]{NIST:DLMF}); i.e.,
\begin{equation*}
\int_0^\infty x \dd F(x) = \frac{( \kappa_d / 2 )^{-1/d}}{d} \int_0^\infty e^{-u} u^{1/d-1} \dd u = ( \kappa_d / 2 )^{-1/d} \gammafcn(1 + 1/d).
\end{equation*}
The asymptotic expansion for large $d$ has been obtained using {\sc Mathematica}. By the same token
\begin{equation*}
\begin{split}
\int_0^R x^2 \dd F( x )
&= R^2 \left( 1 - e^{- ( \kappa_d / 2 ) \, R^d } \right) - 2 \int_0^R x F( x ) \dd x = - R^2 \, e^{-( \kappa_d / 2 ) \, R^d} + 2 \int_0^R x e^{-( \kappa_d / 2 ) \, x^d} \dd x \\
&\to 2 \int_0^\infty x e^{-( \kappa_d / 2 ) \, x^d} \dd x \qquad \text{as $R \to \infty$.}
\end{split}
\end{equation*}
Thus
\begin{equation*}
\int_0^\infty x^2 \dd F( x ) = \frac{2( \kappa_d / 2 )^{2/d}}{d} \int_0^\infty e^{-u} u^{2/d-1} \dd u = ( \kappa_d / 2 )^{-2/d} \gammafcn(1 + 2/d)
\end{equation*}
and
\begin{equation*}
\var( Y ) = \int_0^\infty x^2 \dd F( x ) - \left( \mathbb{E}[ Y ] \right)^2 = \left( \mathbb{E}[ Y ] \right)^2 \left[ \frac{\gammafcn( 1 + 2/d )}{\gammafcn( 1 + 1/d ) \gammafcn( 1 + 1/d )} - 1 \right].
\end{equation*}
\end{proof}

\begin{proof}[Proof of Corollary~\ref{cor:limit.random.separation}]
Let $t \geq 0$ be fixed. Suppose $X_N = \{ \PT{x}_1, \dots, \PT{x}_N \}$ is a set of i.i.d. uniformly chosen points on $\Sph{d}$ and $X_N^\prime$ denotes $\{ \PT{x}_1, \dots, \PT{x}_{N-1} \}$. Then
\begin{equation*}
\prob( \sepDist( X_N ) > t ) = \prob\big( \Theta_{j,N} > t, \, j = 1, \dots, N - 1 \mid \sepDist( X_N^\prime ) > t \big) \cdot \prob\big( \sepDist( X_N^\prime ) > t \big).
\end{equation*}
We estimate the conditional probability. Due to the condition $\sepDist( X_N^\prime ) > t$, all the spherical caps of radius $t/2$ of equal area $A_d(t/2)$ centered at the points in $X_N^\prime$ are disjoint. Hence
\begin{align}
&\prob\big( \Theta_{j,N} > t, \, j = 1, \dots, N - 1 \mid \sepDist( X_N^\prime ) > t \big) \notag \\
&\phantom{equals}\leq \prob\big( \Theta_{j,N} \geq t/2, \, j = 1, \dots, N - 1 \mid \sepDist( X_N^\prime ) > t \big) \notag  \\
&\phantom{equals}= \max\big\{ 0, 1 - ( N - 1 ) A_d(t/2) \big\} \label{eq:upper.bound.prob.} \\
&\phantom{equals}\leq e^{- ( N - 1 ) A_d(t/2)} \notag . 
\end{align}
Since the points in $X_N$ are independently chosen, we also have 
\begin{equation*}
\prob( \sepDist( X_N^\prime ) > t ) = \prob( \sepDist( X_{N-1} ) > t )
\end{equation*}
for a set $X_{N-1}$ of i.i.d. uniformly chosen points on $\Sph{d}$. Thus,
\begin{equation*}
\prob( \sepDist( X_N ) > t ) \leq e^{- ( N - 1 ) A_d(t/2)} \, \prob( \sepDist( X_{N-1} ) > t ).
\end{equation*}
By successive application of this inequality, we arrive at
\begin{align*}
\prob( \sepDist( X_N ) > t ) 
&\leq e^{- \sum_{k=2}^{N-1} k A_d(t/2)} \, \prob( \sepDist( X_{2} ) > t ) \\
&\leq e^{- \sum_{k=2}^{N-1} k A_d(t/2)} e^{-  A_d(t/2)} \\
&\leq e^{- [ (N - 1) N / 2 ] A_d(t/2)}.
\end{align*}
Using the last relation with $t$ replaced by $N^{-2/d} t$, we have the following estimate for $F_N$ given in \eqref{eq:F.N.extreme.law}, 
\begin{equation*}
1 - F_N( t ) = \prob( \sepDist( X_N ) > N^{-2/d} t ) \leq e^{- [ (N - 1) N / 2 ] A_d(N^{-2/d} t/2)}.
\end{equation*}
Using $\rho = 2 \sin( N^{-2/d} t/4 )$ in the surface area formula \eqref{eq:sigma_d.C.rho}, we get 
\begin{align*}
A_d(N^{-2/d} t/2) = \sigma_d( C_\rho ) 
&\geq \kappa_d \rho^d = \kappa_d \left( 2 \sin( N^{-2/d} t/4 ) \right)^d \\
&\geq \kappa_d \left( N^{-2/d} t / 4  \right)^d = \kappa_d 2^{-2d} N^{-2} \, t^d,
\end{align*}
which implies
\begin{equation*}
1 - F_N( t ) \leq e^{- [ (N - 1) N / 2 ] \kappa_d 2^{-2d} N^{-2} \, t^d} \leq e^{-c_d \, t^d}
\end{equation*}
for some $c_d > 0$ depending only on $d$. Since the right-hand side above as a function in $t$ is integrable on $(0,\infty)$, the Lebesgue Dominated Convergence Theorem yields
\begin{align*}
\lim_{N \to \infty} \mathbb{E}[ N^{2/d} \, \sepDist( X_N ) ] 
&= \lim_{N \to \infty} \int_0^\infty \left( 1 - F_N( t ) \right) \dd t \\
&= \int_0^\infty \lim_{N \to \infty} \left( 1 - F_N( t ) \right) \dd t \\
&= \int_0^\infty \left( 1 - F( t ) \right) \dd t \\
&= \left( \frac{\kappa_d}{2} \right)^{-1/d} \gammafcn( 1 + 1 / d ),
\end{align*}
the last step following from Lemma~\ref{lem:aux.expected}.
\end{proof}

\begin{proof}[Proof of Proposition~\ref{prop:separation.probability}]
Let $C( \PT{x}, \eps )$ denote the spherical cap on $\Sph{d}$ of center $\PT{x} \in \Sph{d}$ and geodesic radius $\eps$. If we think of selecting the random points $\PT{X}_1, \dots, \PT{X}_N$ one after the other, we naturally write
\begin{equation*}
\begin{split}
\prob( \MARKED{red}{\sepDist( X_N )} \geq \eps )
&= \prob( \PT{X}_2 \not\in C( \PT{X}_1, \eps ) \mid \PT{X}_1 ) \\
&\phantom{=\pm}\times \prob( \PT{X}_3 \not\in C( \PT{X}_1, \eps ) \cup C( \PT{X}_2, \eps ) \mid \PT{X}_1, \PT{X}_2 ) \\
&\phantom{=\pm}\cdots \\
&\phantom{=\pm}\times \prob( \PT{X}_N \not\in C( \PT{X}_1, \eps ) \cup \cdots \cup C( \PT{X}_{N-1}, \eps ) \mid \PT{X}_1, \dots, \PT{X}_{N-1} ).
\end{split}
\end{equation*}
By the uniformity of the distribution,
\begin{align*}
&\prob( \PT{X}_{k+1} \not\in C( \PT{X}_1, \eps ) \cup \cdots \cup C( \PT{X}_{k}, \eps ) \mid \PT{X}_1, \dots, \PT{X}_{k} ) \\
&\phantom{equals}= 1 - \sigma_d\big( C( \PT{X}_1, \eps ) \cup \cdots \cup C( \PT{X}_{k}, \eps ) \big) \geq 1 - k \, A_d(\eps), \qquad k = 1, \dots, N - 1,
\end{align*}
and hence
\begin{equation*}
\prob( \MARKED{red}{\sepDist( X_N )} \geq \eps ) \geq \prod_{k=1}^{N-1} \left( 1 - k \, A_d(\eps) \right) \geq 1 - \binom{N}{2} \, A_d(\eps),
\end{equation*}
where the last inequality follows easily by induction. This establishes the first inequality in the \MARKED{red}{lower bound of $\prob( \MARKED{red}{\sepDist( X_N )} \geq \eps )$ in \eqref{eq:separation.probability.A}}. 
The second \MARKED{red}{lower bound} follows then from the well-known fact (cf., e.g., \cite{KuSa1998}) that
\begin{equation*}
A_d(\eps) = \sigma_d( C( \PT{x}, \eps ) ) \leq \kappa_d  \left( 2 \sin \frac{\eps}{2} \right)^d,
\end{equation*}
where the Euclidean radius $\rho$ and the geodesic radius $\eps$ are related by $\rho = 2 \sin( \eps / 2 )$.
\MARKED{red}{The upper bound of $\prob( \MARKED{red}{\sepDist( X_N )} \geq \eps )$  in \eqref{eq:separation.probability.A} follows from the proof of Corollary~\ref{cor:limit.random.separation} (cf.~\eqref{eq:upper.bound.prob.}).} 
Relation~\ref{eq:separation.probability.B} follows when setting $\eps = \delta \, N^{-2/d}$ in \eqref{eq:separation.probability.A}. Relation~\ref{eq:separation.probability.C} follows when setting $\delta = C_d$ ($C_d$ from Corollary~\ref{cor:limit.random.separation}) in \eqref{eq:separation.probability.B}. The asymptotic expansion of $( \gammafcn( 1 + 1 / d ) )^d$ can be obtained with the help of {\sc Mathematica}. The monotonicity of the right-hand side of \eqref{eq:separation.probability.C} follows from
\begin{equation*}
\frac{\dd }{\dd x} \left\{ 1 - \left[ \gammafcn\Big( 1 + \frac{1}{x} \Big) \right]^x \right\} = \left[ \gammafcn\Big( 1 + \frac{1}{x} \Big) \right]^x \left\{ \frac{\digammafcn( 1 + \frac{1}{x} )}{x} - \log \gammafcn\Big( 1 + \frac{1}{x} \Big) \right\}
\end{equation*}
and the observation that the right expression in braces is positive, because it is strictly decreasing as $x$ grows (since having the negative derivative $-\digammafcn_1( 1 + 1/x ) / x^3$) and it has the asymptotic expansion $\frac{\pi^2}{12 \, d^2} + \mathcal{O}(\frac{1}{d^3})$ as $d \to \infty$ with positive dominating term. Here, $\digammafcn( x ) \DEF \gammafcn^\prime( x ) / \gammafcn( x )$ is the digamma function and $\psi_1( x )$ is its derivative.
\end{proof}

\begin{proof}[Proof of Proposition~\ref{prop:limit.random.separation}]
Let $f( \delta )$ denote the lower bound of the expected value in \eqref{eq:expected.normalized.Theta.min.lower.bound}. Then
\begin{equation*}
f^\prime( \delta ) = 1 - \frac{1}{2} \, \frac{d+1}{d} \, \frac{\omega_{d-1}}{\omega_d} \delta^d
\end{equation*}
and $f$ thus has a unique minimum at $\delta = \delta^*$ with
\begin{equation*}
\delta^* \DEF \left( \frac{1}{2} \, \frac{d+1}{d} \, \frac{\omega_{d-1}}{\omega_d} \right)^{-1/d} \in \left( 0, \left[ \kappa_d \right]^{-1/d} \right) \qquad \text{for $d \geq 1$}
\end{equation*}
and assumes the value
\begin{equation*}
f( \delta^* ) = \left( \frac{1}{2} \, \frac{d+1}{d} \, \frac{\omega_{d-1}}{\omega_d} \right)^{-1/d} \left( 1 - \frac{1}{d+1} \right) = C_d \, \frac{\left( d + 1 \right)^{-1/d}}{\gammafcn( 2 + 1 / d )}.
\end{equation*}
The \MARKED{red}{asymptotic} expansion of $( d + 1)^{-1/d} / \gammafcn( 2 + 1/d )$ can be obtained with the help of {\sc Mathematica}.
\end{proof}

\section{Comparison with deterministic point sets}
\label{sec:comparison}

\subsection{Point sets}

Many different point sets have been studied and used in applications,
especially on $\Sph{2}$. The aim of this section is to give some idea of
the distribution of both the separation distance and the scaled hole radii
as a contrast to those for uniformly distributed random points. We
consider the following point sets $\PTCfg_N$.
\begin{itemize}
\item \textbf{RN} Pseudo-random points, uniformly distributed on the
    sphere.
\item \textbf{ME} Points chosen to minimize the Riesz $s$-energy for
    $s = 1$ (Coulomb energy)
\begin{equation*}
\sum_{j=1}^N \sum_{\substack{i=1\\ i\neq j}}^N \frac{1}{\left| \PT{x}_j - \PT{x}_i \right|^s}.
\end{equation*}
As $s\to \infty$ the minimal $s$-energy points approach best
separation.
\item \textbf{MD} Points chosen to maximize the determinant for polynomial interpolation \cite{SlWo2004}.
\item \textbf{FI} Fibonacci points, with $\Phi = (1 + \sqrt{5})/2$ and
    spherical coordinates $(\theta_j, \phi_j)$:
\begin{equation*}
  z_j = 1 - \tfrac{2j-1}{N}, \  \theta_j = \cos^{-1}(z_j), \  \phi_j = \tfrac{\pi}{\Phi} (N+1-2j) \mod 2\pi,
 \quad j = 1,\ldots,N.
\end{equation*}
\item \textbf{SD} Spherical $t$-designs with $N = \lceil (t+1)^2/2 \rceil + 1$ points for $t = 45$, so $N = 1059$.
\item \textbf{CV} Points chosen to minimize the covering radius $\alpha(\PTCfg_N)$ (best covering).
\item \textbf{PK} Points chosen to maximize the separation $\sepDist(\PTCfg_N)$ (best packing).
\item \textbf{PU} Points that maximize the $s$-polarization
\begin{equation*}
  \min_{\PT{x}\in\Sph{2}} \sum_{j=1}^N  \frac{1}{\left| \PT{x} - \PT{x}_j \right|^s}
\end{equation*}
for $s = 4$. As $s \to \infty$ the maximum polarization points approach best covering.
\end{itemize}
All the point sets that are characterised by optimizing a criterion are
faced with the difficulty of many local optima. Thus, for even modest
values of $N$, these point sets have objective values near, but not
necessarily equal to, the global optimum.

\subsection{Covering}

For a variety of point sets, Figure~\ref{fig:HolRadE} illustrates the
distribution of the $f_2 = 2N - 4$ (so $f_2 = 2044$ for $N = 1024$) scaled
(by $N^{1/2}$) Euclidean hole radii.
When changing from random to non-random points,
the distribution of hole radii changes significantly.
Figure~\ref{fig:HolRadE} has the center of the distribution at the lower
end of the range of realized hole radii, whereas for points with good
covering property the center of the distribution is, as expected, at the
upper end of the values of hole-radii. However, one should note that,
compared with the random setting, the distribution for hole radii for
point sets having small Coulomb energy or large polarization
are remarkable localized.

\begin{figure}[ht]
\begin{center}
\includegraphics[scale=0.35]{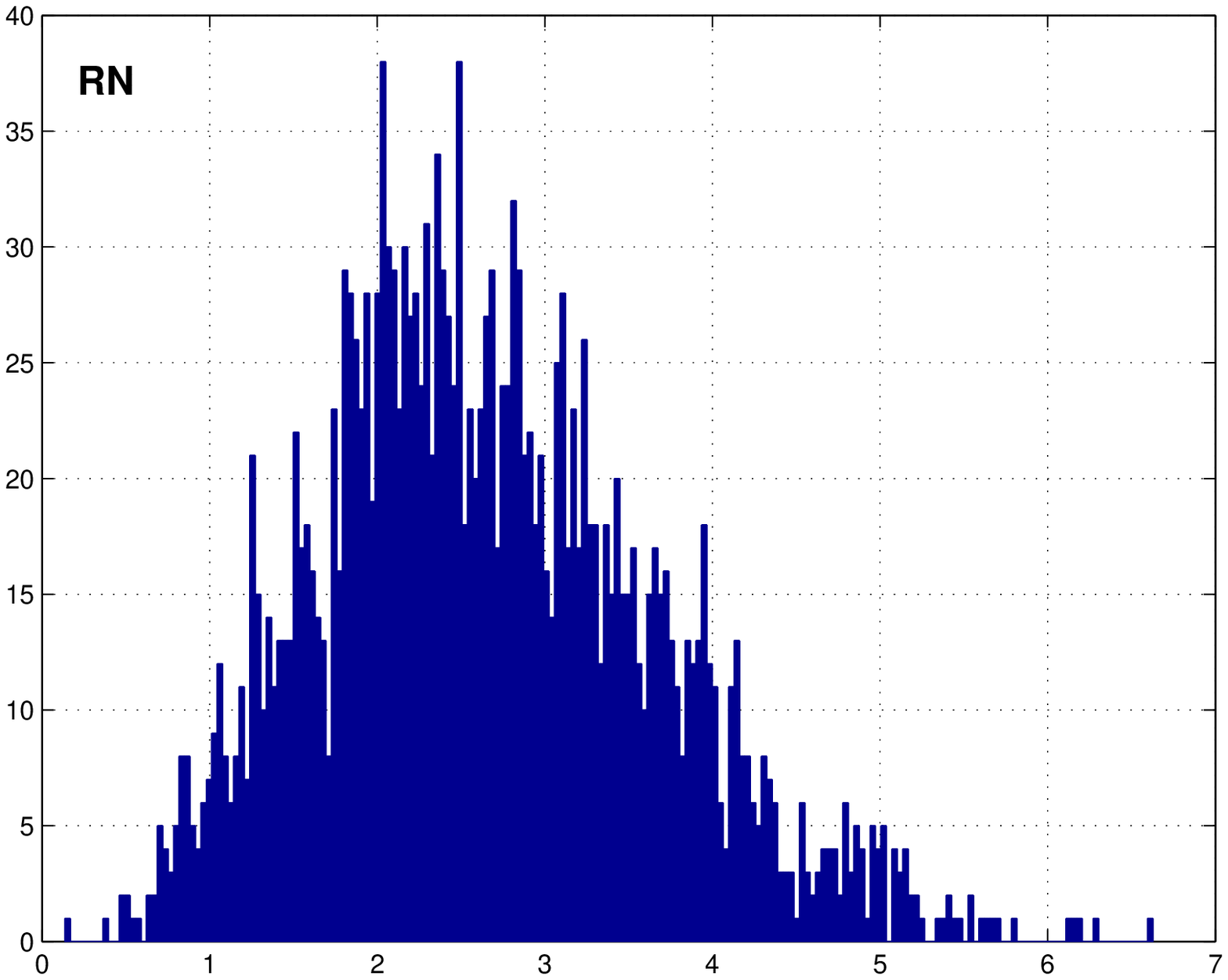} \quad
\includegraphics[scale=0.35]{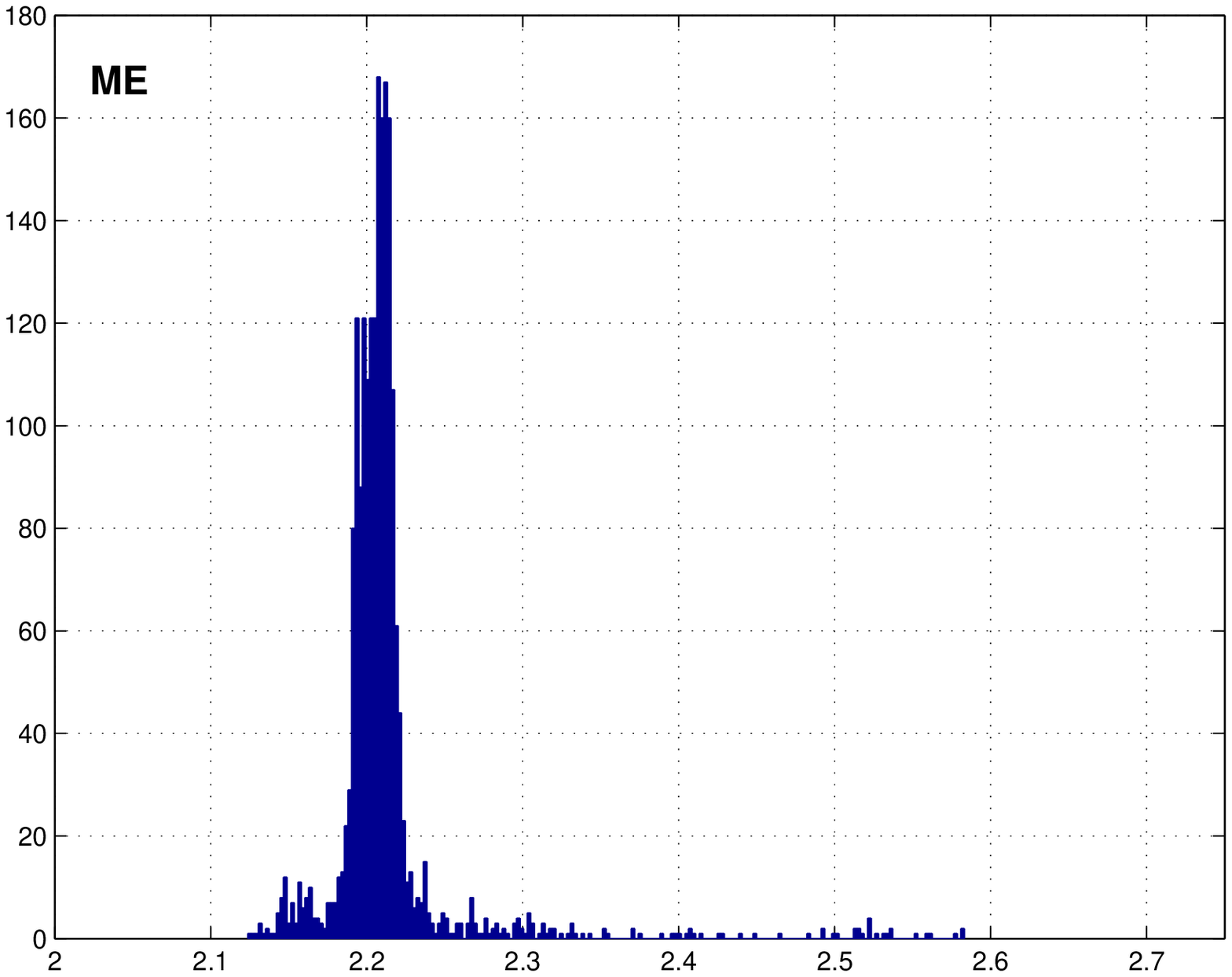} \\
\includegraphics[scale=0.35]{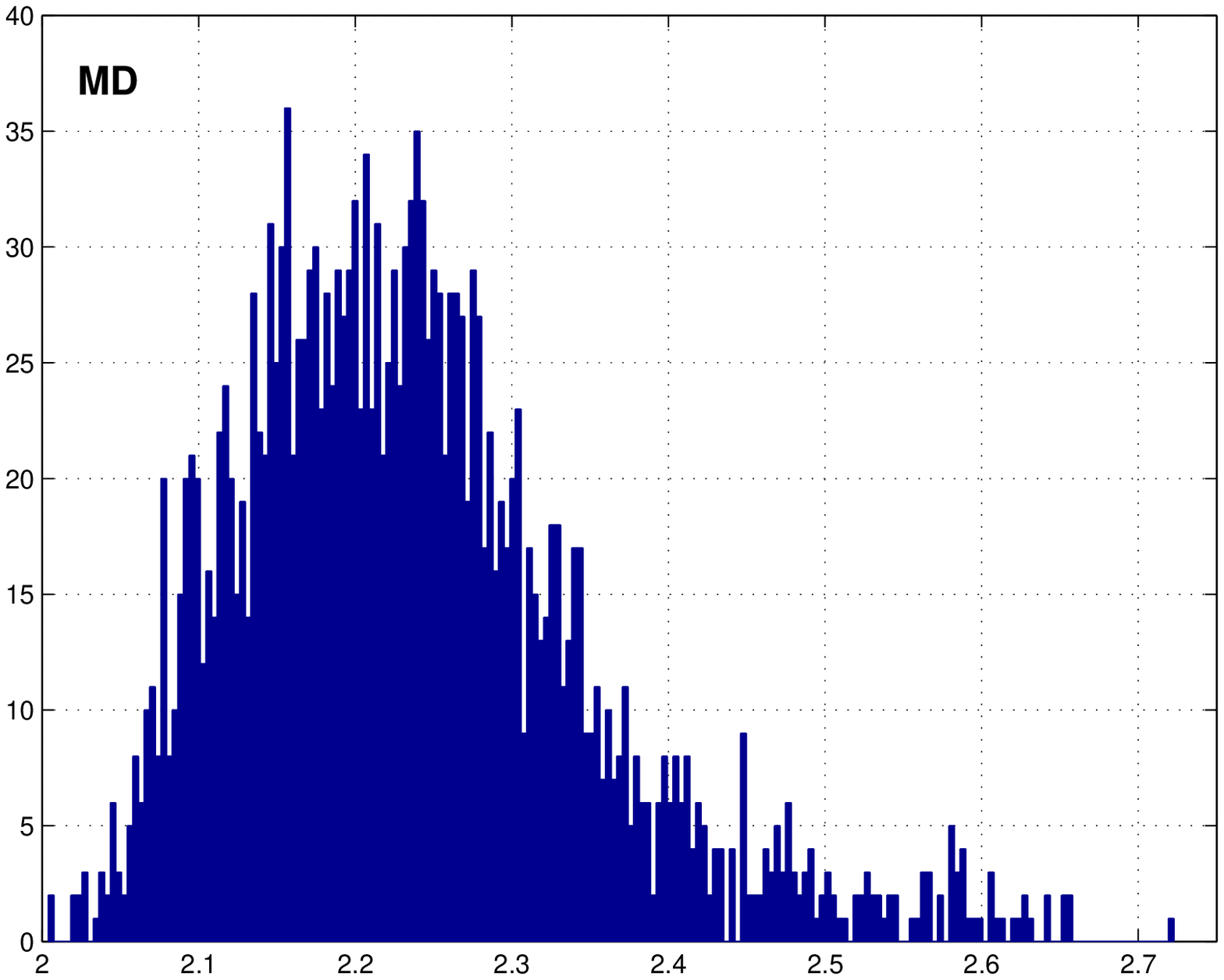} \quad
\includegraphics[scale=0.35]{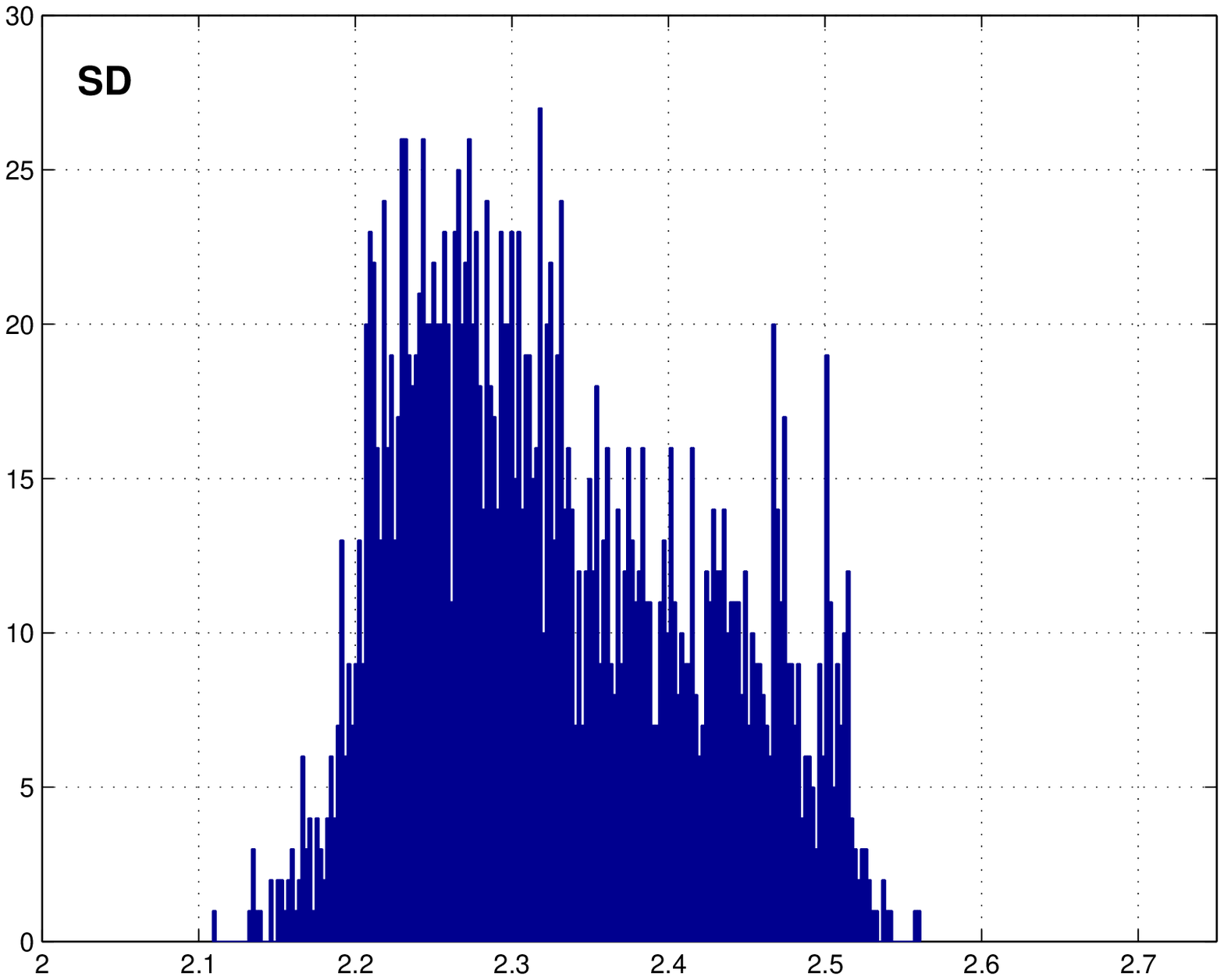} \\
\includegraphics[scale=0.35]{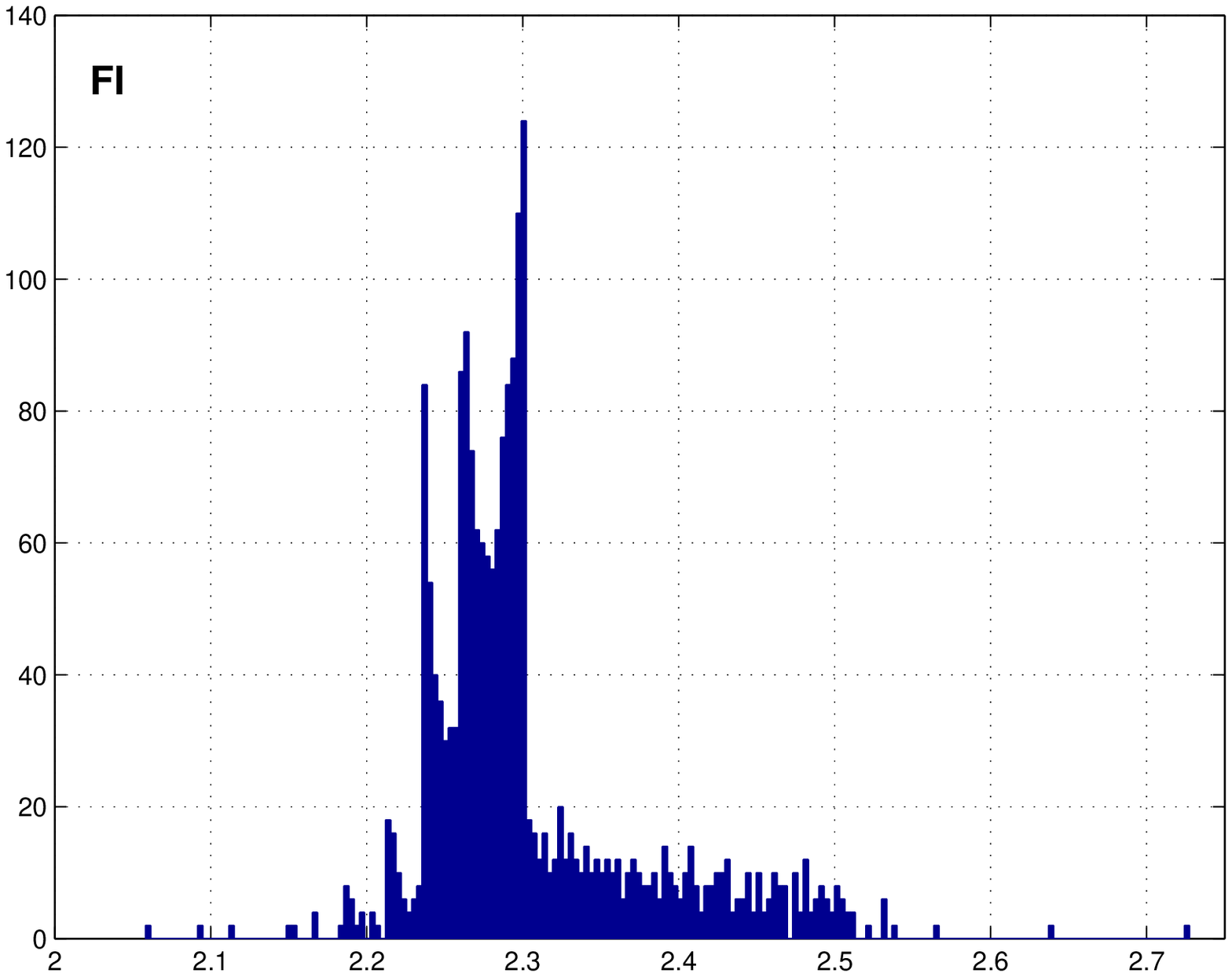} \quad
\includegraphics[scale=0.35]{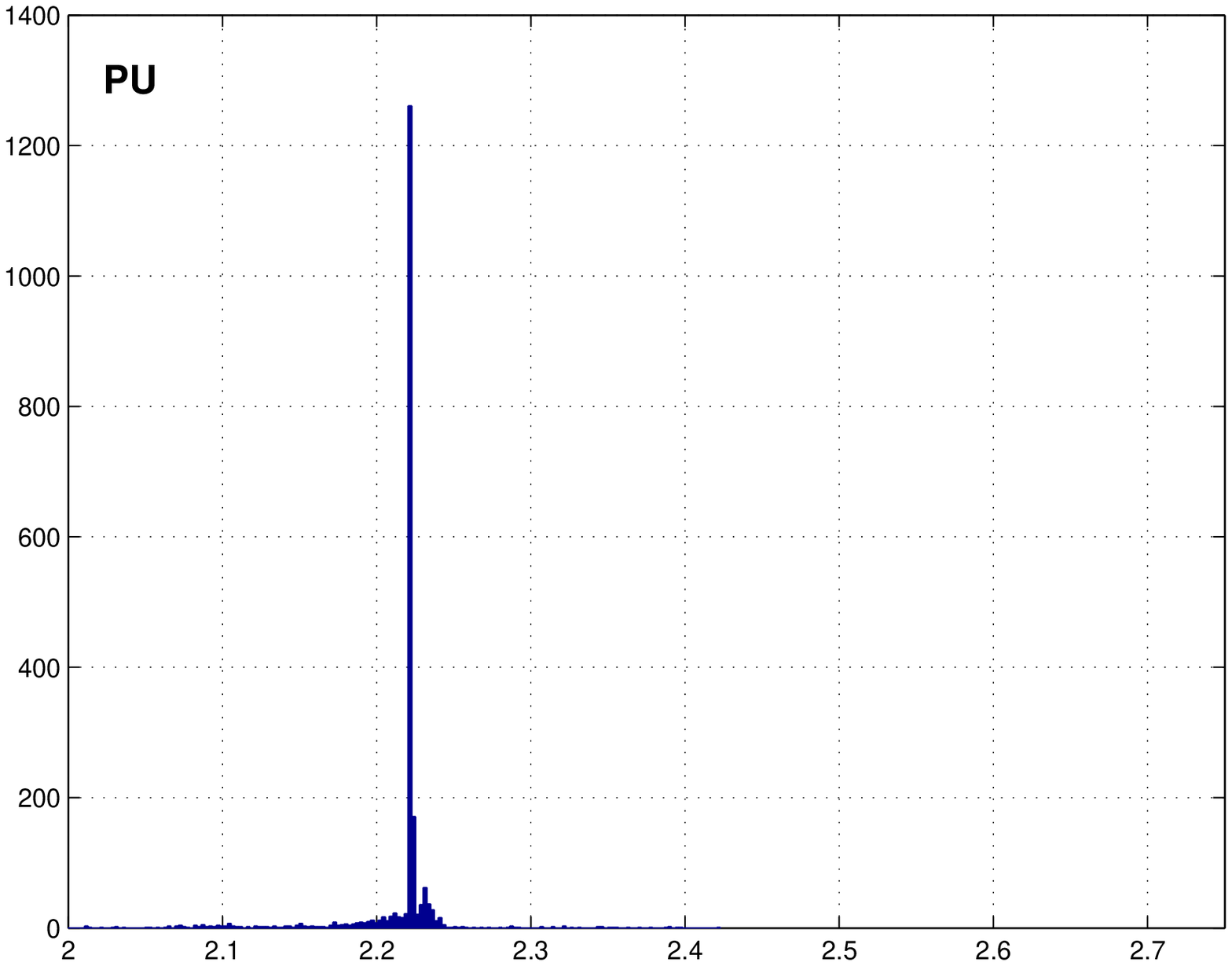} \\
\includegraphics[scale=0.35]{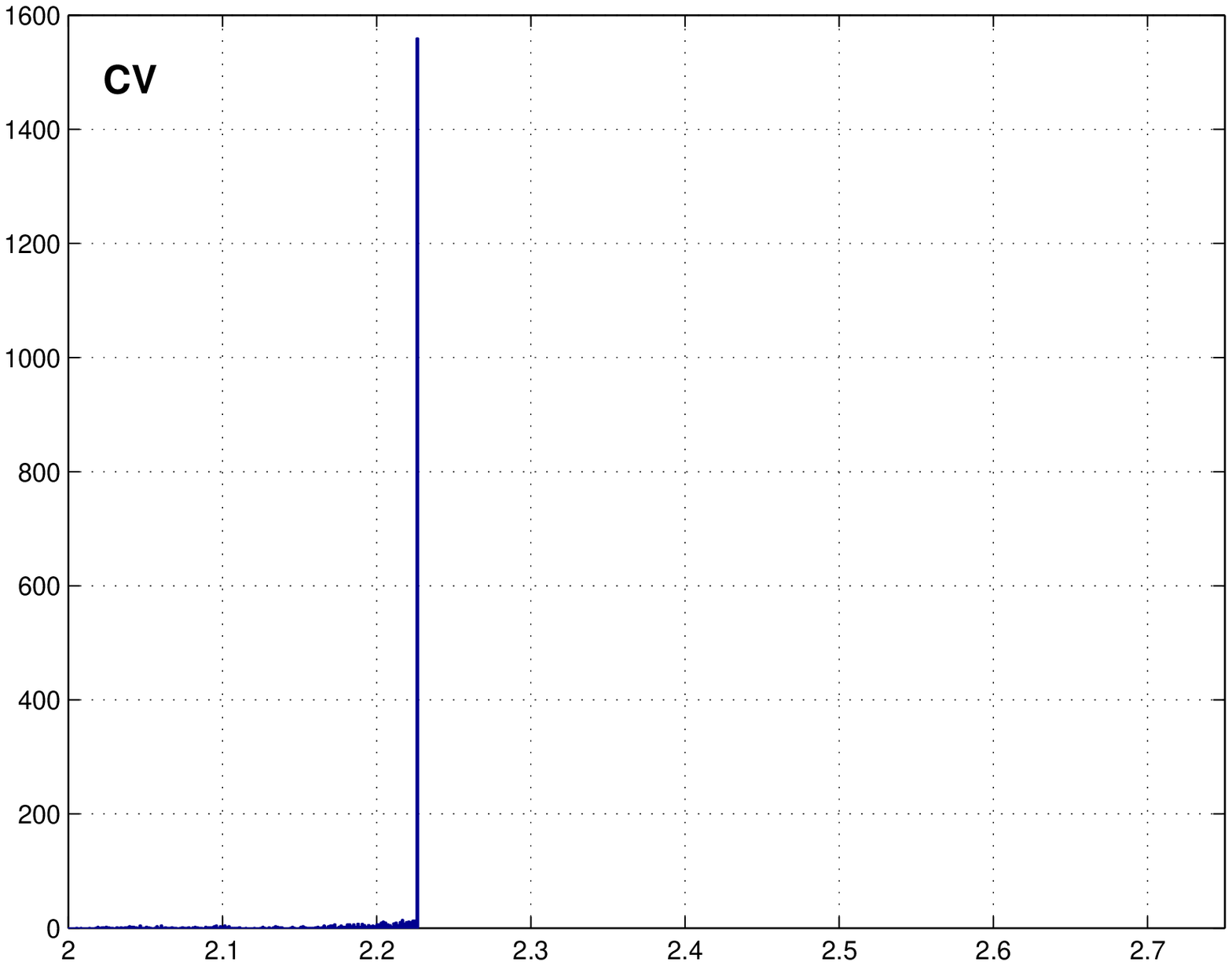} \quad
\includegraphics[scale=0.35]{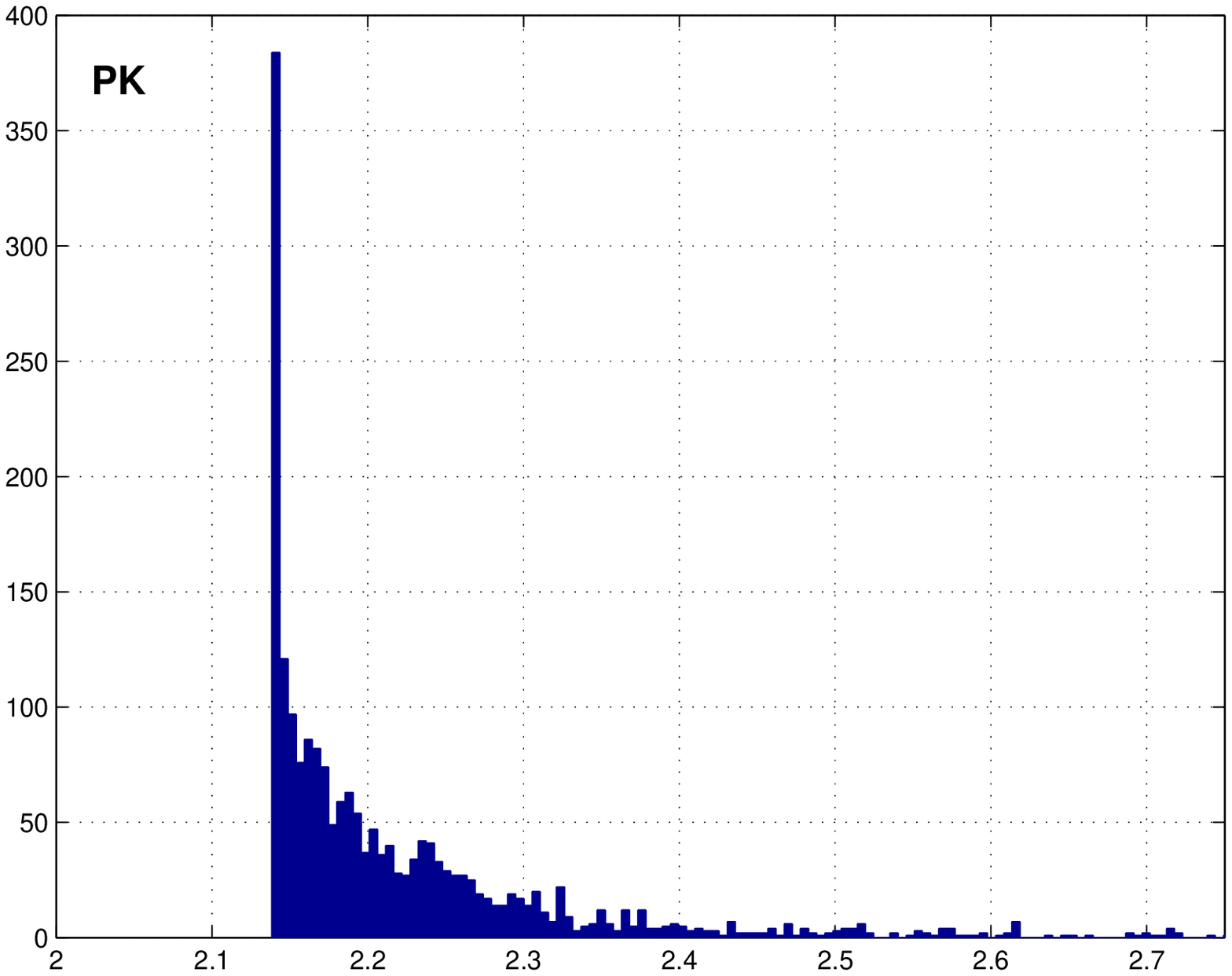} \\
\caption{Histograms of the scaled Euclidean hole radii for the sets of points in Figure~\ref{fig:SepDisE}}
\label{fig:HolRadE}
\end{center}
\end{figure}

\subsection{Separation}

For a variety of point sets, Figure~\ref{fig:SepDisE} illustrates the
distribution of the pairwise Euclidean separation distances for which
$|\PT{x}_i - \PT{x}_j| \leq 1$ \MARKED{red}{for the same sets of points as in
Figures~\ref{fig:HolRadE}.} The vertical lines denote, as in
\cite[Figure 1]{BrHaSa2012}, the hexagonal lattice distances scaled so
that the smallest distance coincides with the best packing distance.
\begin{figure}[ht]
\begin{center}
\includegraphics[scale=0.325]{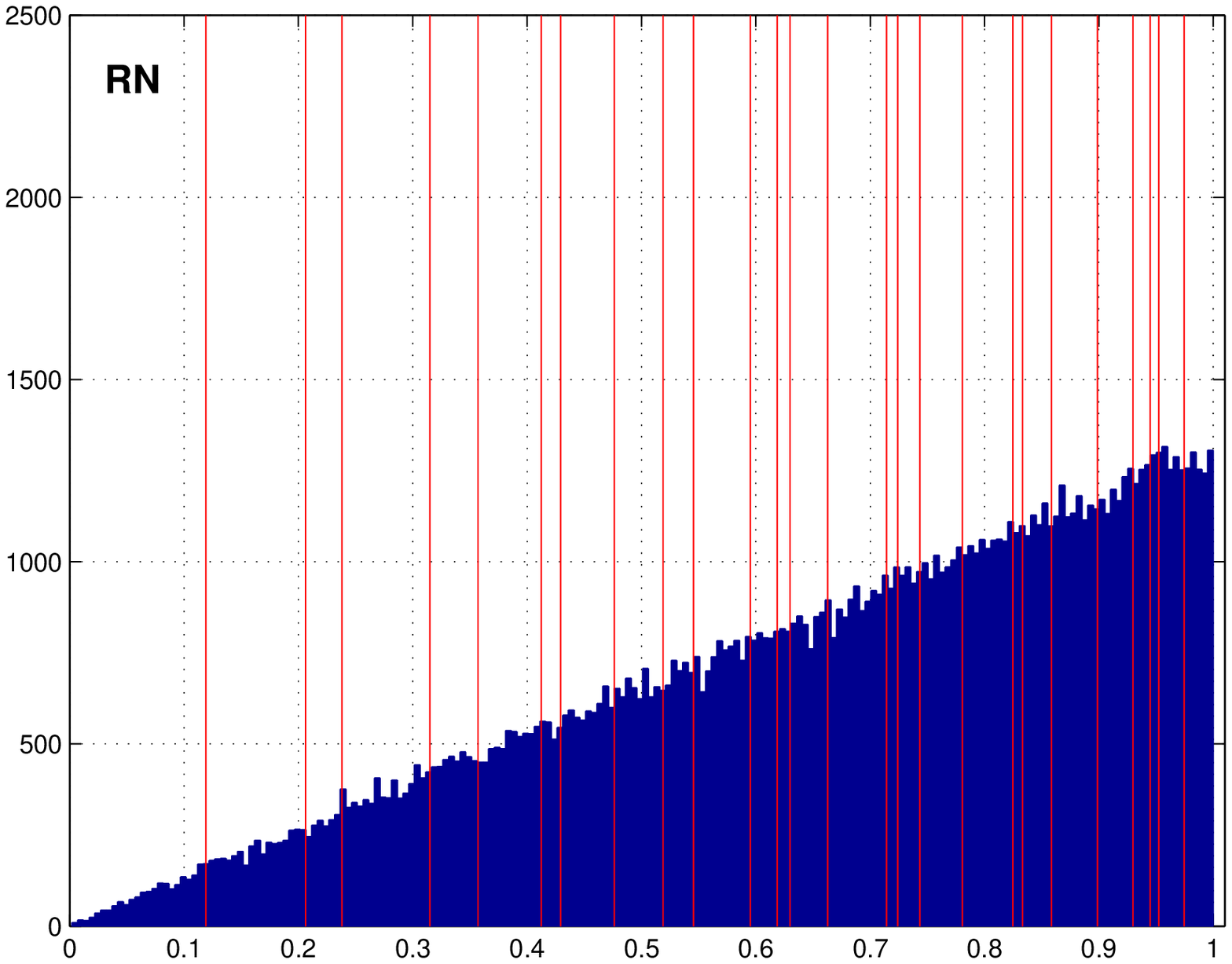} \quad
\includegraphics[scale=0.325]{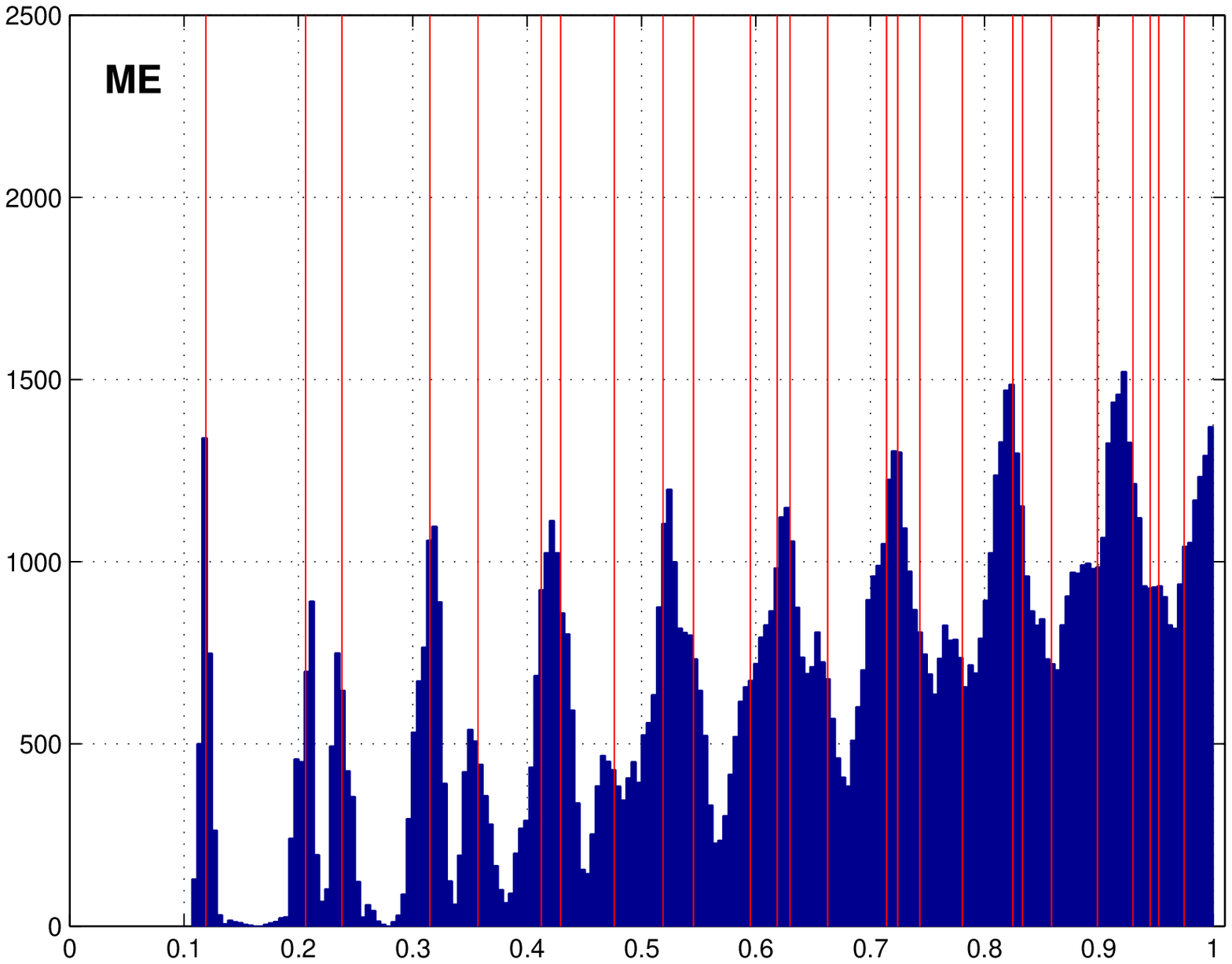} \\
\includegraphics[scale=0.325]{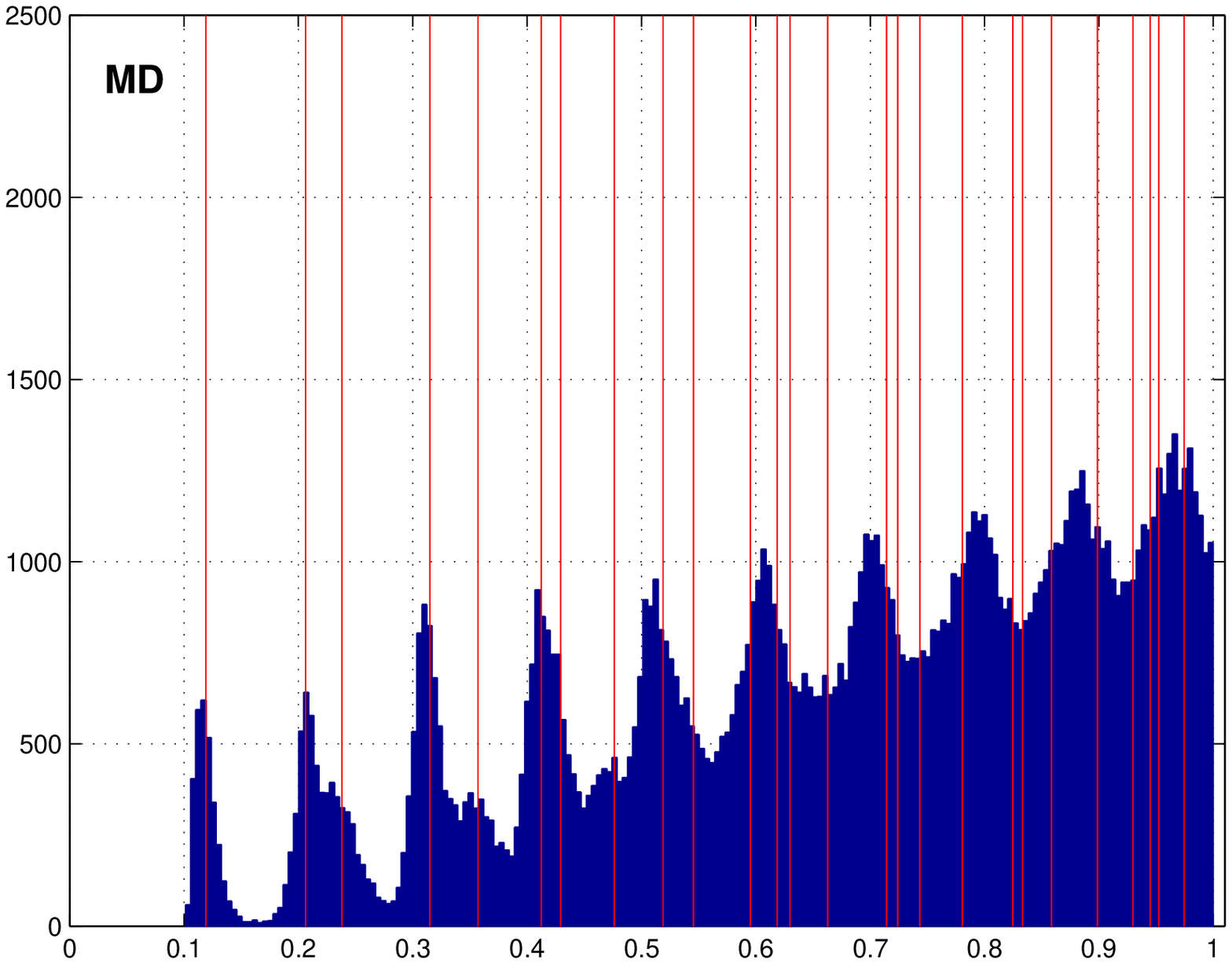} \quad
\includegraphics[scale=0.325]{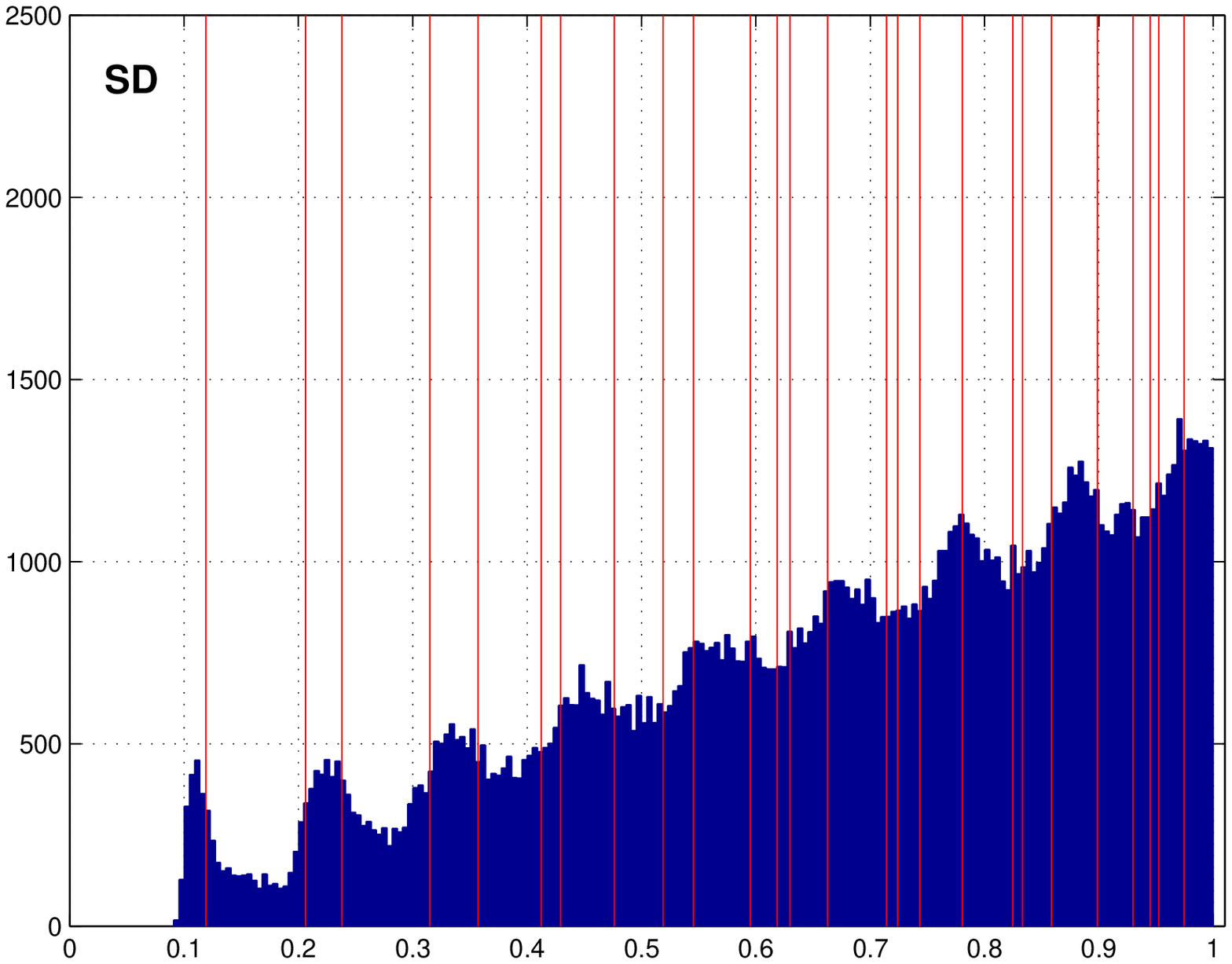} \\
\includegraphics[scale=0.325]{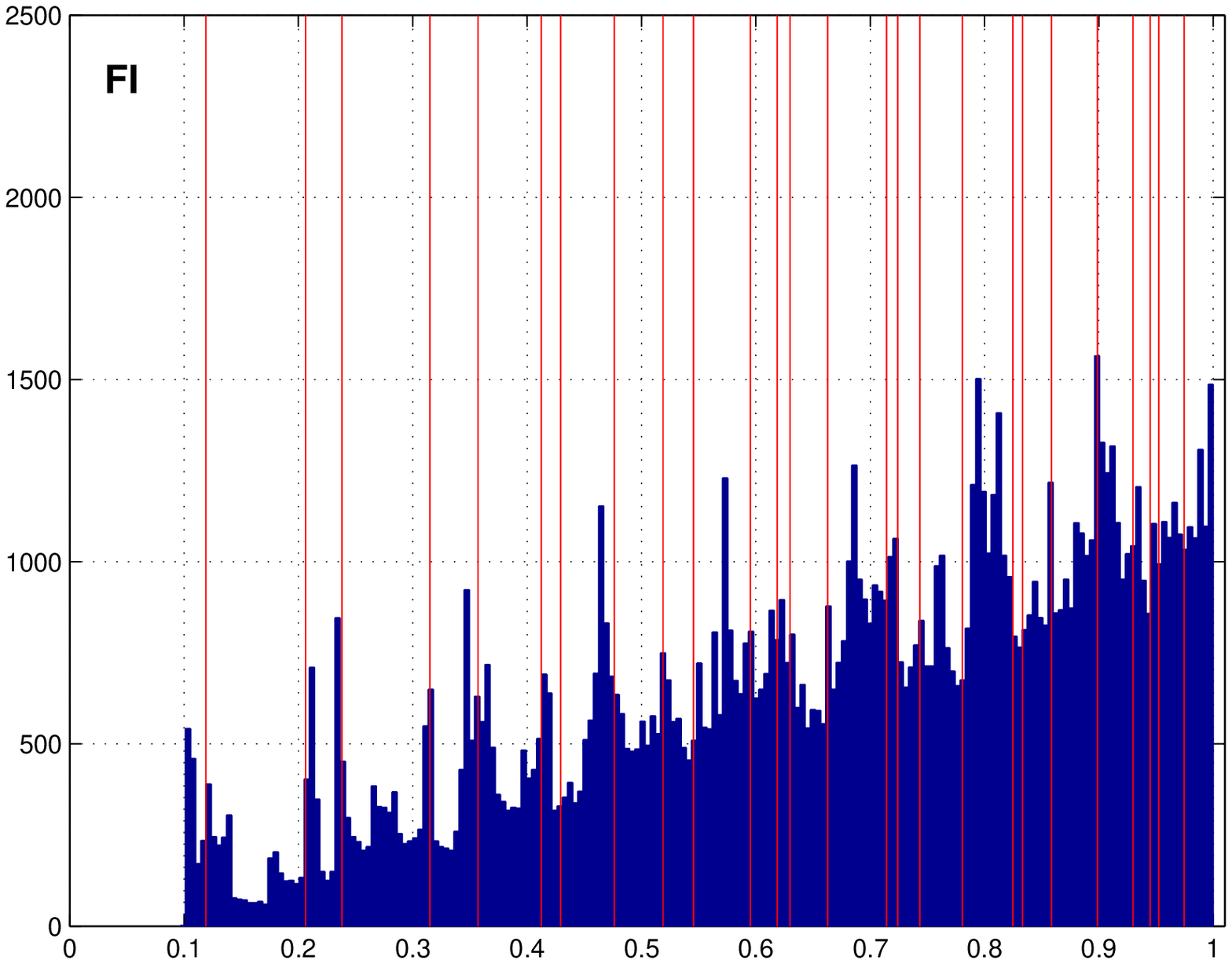} \quad
\includegraphics[scale=0.325]{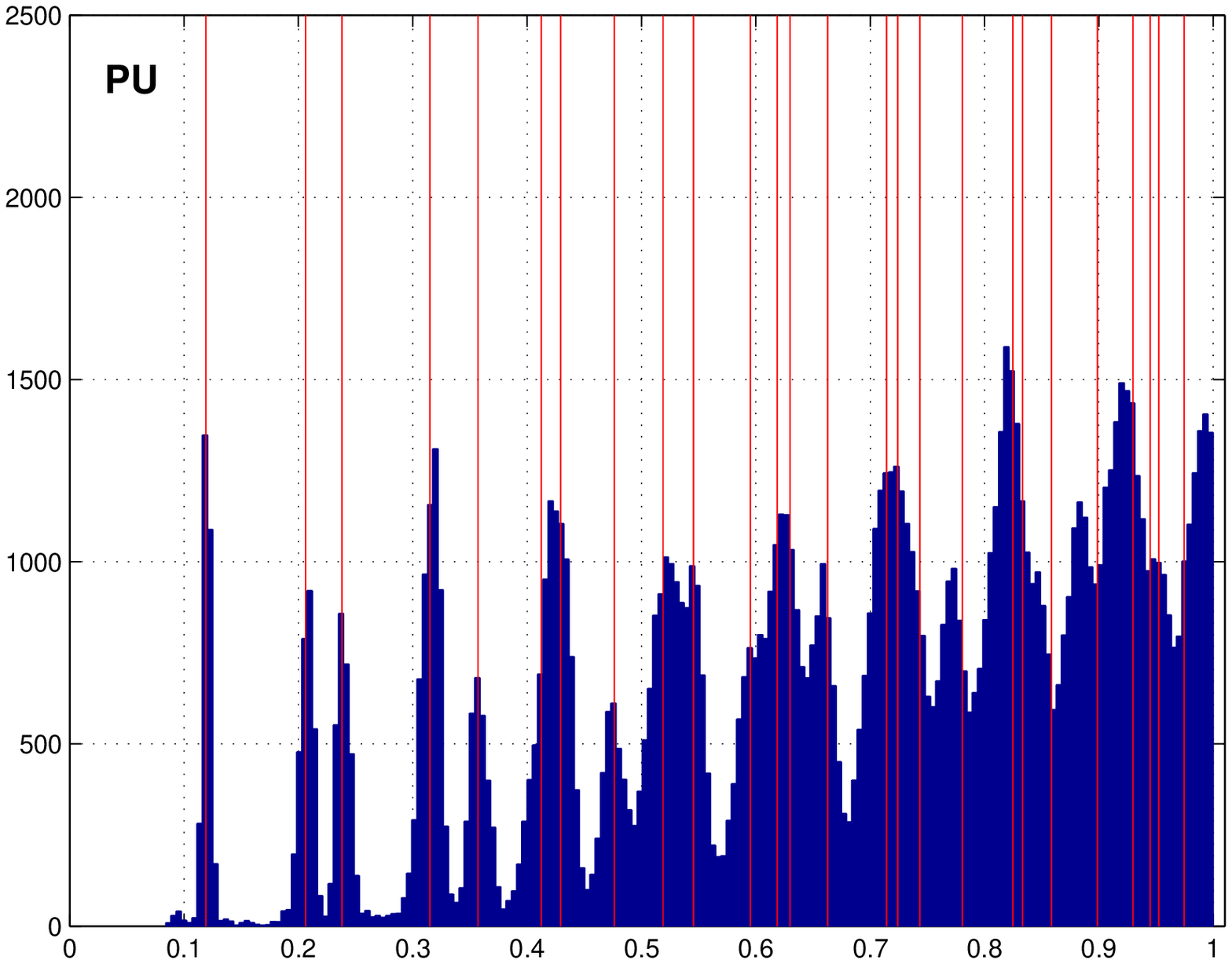} \\
\includegraphics[scale=0.325]{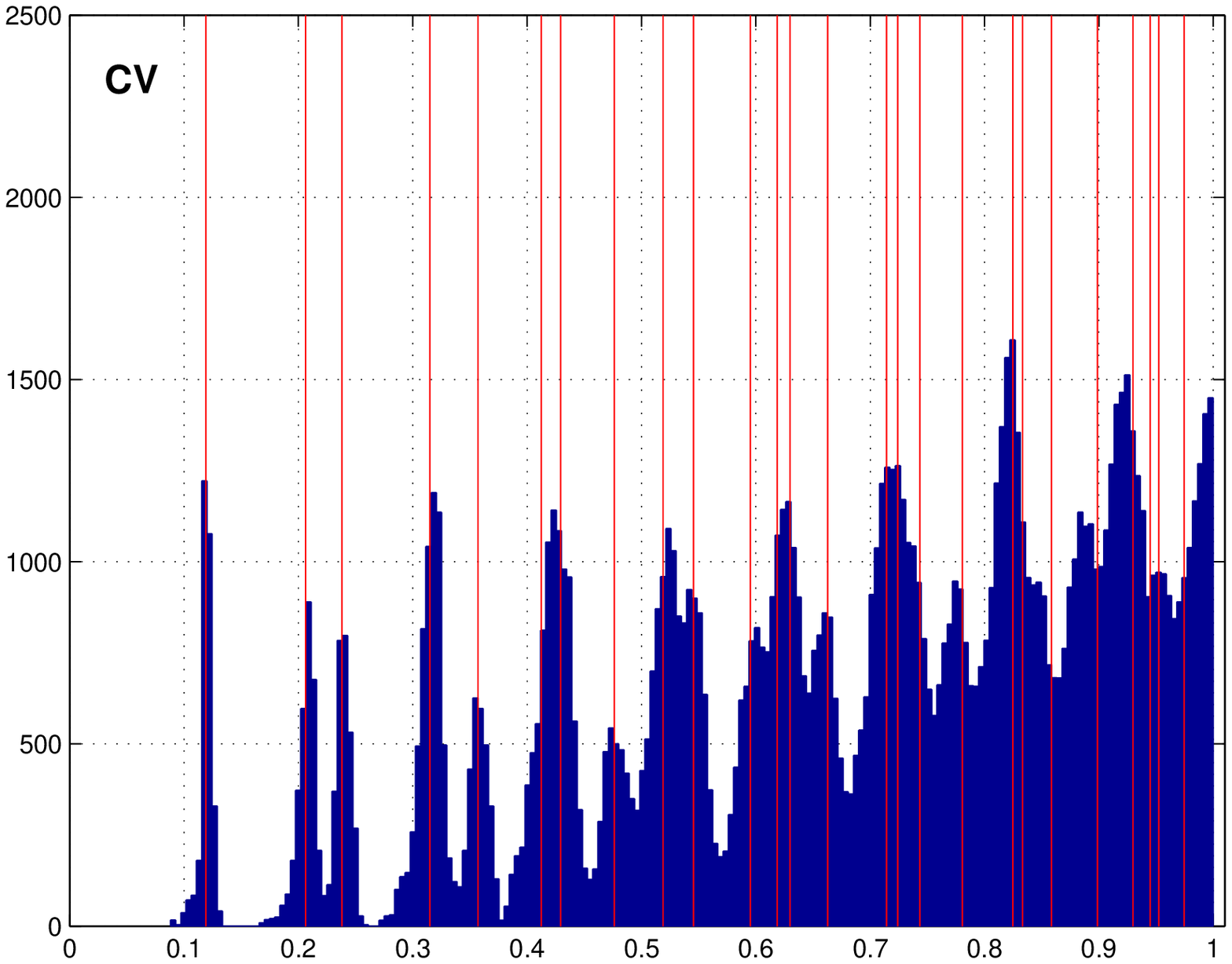} \quad
\includegraphics[scale=0.325]{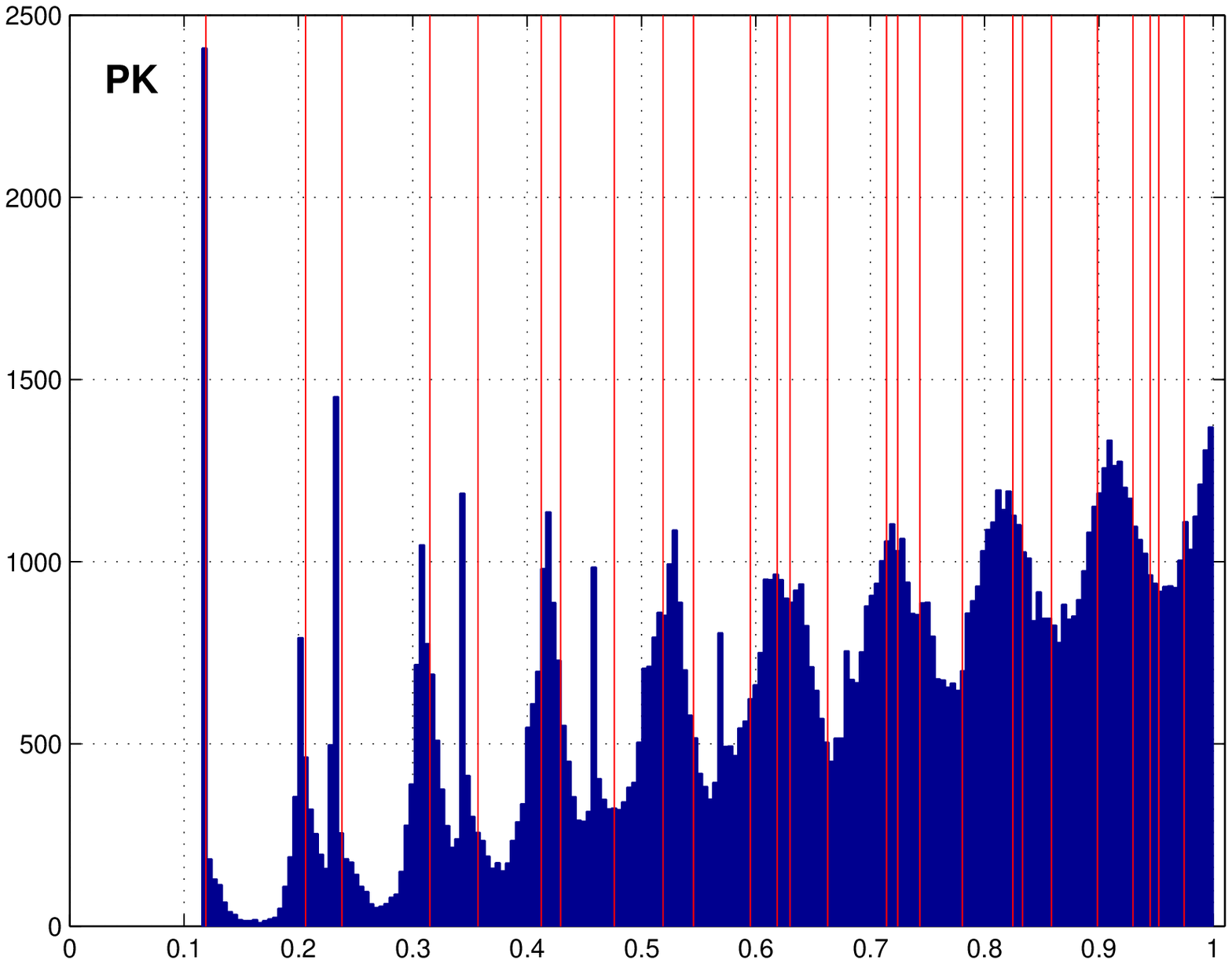} \\
\caption{Histograms of the Euclidean separation distances below $1$ for various sets of $N = 1024$ points
(except for SD where $N = 1059$).
The point set is indicated by  the two letters in the top left corner of each subplot.
The vertical lines are the hexagonal lattice distances scaled so the smallest distance is the best
packing distance}
\label{fig:SepDisE}
\end{center}
\end{figure}

\vspace{2em}

\MARKED{red}{
{\bf Acknowledgements:} The authors are grateful to two anonymous referees for their comments. We also wish to express our appreciation to Tiefeng Jiang and Jianqing Fan for very helpful discussions concerning Corollary~\ref{cor:limit.random.separation}.
}

\bibliographystyle{abbrv}
\bibliography{bibliography}
\end{document}